\documentclass[11pt,a4paper]{article}
\usepackage{color}
\usepackage{enumerate}
\usepackage{comment}
\usepackage[english]{babel}
\usepackage{exscale,times}

\usepackage{amssymb}
\usepackage{amsmath}
\usepackage{psfrag}
\usepackage{ifthen}
\usepackage{graphicx}

\usepackage{fullpage}

\newenvironment{proof}[1][]
{
   \noindent\textrm{\bf Proof%
   \ifthenelse{\equal{#1}{}}{:}{~#1:} }\rm
}
{
   \hfill$\square$
   \bigskip
}

\newtheorem{theorem}{Theorem}[section]
\newtheorem{proposition}[theorem]{Proposition}
\newtheorem{lemma}[theorem]{Lemma}
\newtheorem{corollary}[theorem]{Corollary}
\newtheorem{definition}[theorem]{Definition}
\newtheorem{remark}[theorem]{Remark}
\newtheorem{example}[theorem]{Example}
\newcommand{\eex}{\hbox{}\hfill\rule{0.8ex}{0.8ex}}
\newcommand{\eremk}{\eex}

\newcommand{\R}[0]{\mathbb{R}}

\newcommand{\N}[0]{\mathbb{N}}

\renewcommand{\H}[0]{\mathcal{H}}

\newcommand{\ol}[1]{\overline{#1}}

\newcommand{\ra}[0]{\rightarrow}

\newcommand{\abs}[1]{\left\vert #1 \right\vert}
\newcommand{\norm}[1]{\left\| #1 \right\|}

\newcommand{\bea}{\begin{eqnarray*}}
\newcommand{\eea}{\end{eqnarray*}}
\newcommand{\bean}{\begin{eqnarray}}
\newcommand{\eean}{\end{eqnarray}}
\newcommand{\been}{\begin{equation}}
\newcommand{\een}{\end{equation}}
\newcommand{\bee}{\begin{equation*}}
\newcommand{\ee}{\end{equation*}}

\newcommand{\bdefi}{\begin{definition}}
\newcommand{\edefi}{\end{definition}}
\newcommand{\bthe}{\begin{theorem}}
\newcommand{\ethe}{\end{theorem}}
\newcommand{\bpro}{\begin{proof}}
\newcommand{\epro}{\end{proof}}
\newcommand{\blem}{\begin{lemma}}
\newcommand{\elem}{\end{lemma}}
\newcommand{\bcor}{\begin{corollary}}
\newcommand{\ecor}{\end{corollary}}
\newcommand{\brem}{\begin{remark}}
\newcommand{\erem}{\end{remark}}
\newcommand{\bpropo}{\begin{proposition}}
\newcommand{\epropo}{\end{proposition}}

\newcommand{\bsatz}{\begin{satz}}
\newcommand{\esatz}{\end{satz}}

\newcommand{\bbsp}{\begin{beispiel}\rm}
\newcommand{\ebsp}{\end{beispiel}}

\newcommand{\bbem}{\begin{bemerkung}}
\newcommand{\ebem}{\end{bemerkung}}

\newcommand{\bpm}{\begin{pmatrix}}
\newcommand{\epm}{\end{pmatrix}}

\renewcommand{\div}{\rm \div}

\renewcommand{\epsilon}{\varepsilon}

\newcommand{\supp}[1]{\text{supp}(#1)}
\renewcommand{\supp}{\operatorname*{supp}}

\newcommand{\skp}[1]{\left< #1 \right>}
\newcommand{\T}{\mathcal{T}}

\newcommand{\unjump}{[\partial_{n}u]}
\newcommand{\vnjump}{[\partial_{n}v]}

\newcommand{\vjump}{[\gamma_0 v]}

\newcommand{\lefttriplenorm}{\ensuremath{\left| \! \left| \! \left|}}
\newcommand{\righttriplenorm}{\ensuremath{\right| \! \right| \! \right|}}
\newcommand{\triplenorm}[1]{\lefttriplenorm #1 \righttriplenorm}

\newcommand{\level}{{\rm level}}

\newcommand{\Pfar}{P_{\rm far}}

\newcommand{\diam}{{\rm diam}}
\newcommand{\dist}{{\rm dist}}
\newcommand{\curl}{{\rm curl}}

\newcommand{\Win}{\boldsymbol{\mathcal W}^{-1}}

\numberwithin{equation}{section}

\title{}
\author{}

\begin{document}


\hskip 10 pt

\begin{center}
{\fontsize{14}{20}\bf Existence of $\H$-matrix approximants to the inverse of 
BEM matrices: the hyper-singular integral operator}
\end{center}

\begin{center}
\textbf{Markus Faustmann, Jens Markus Melenk, Dirk Praetorius}\\
\bigskip
{Institute for Analysis and Scientific Computing}\\
Vienna University of Technology\\
Wiedner Hauptstr. 8-10, 1040 Wien, Austria\\markus.faustmann@tuwien.ac.at, melenk@tuwien.ac.at, dirk.praetorius@tuwien.ac.at\\
\bigskip
\end{center}

\begin{abstract}
We consider discretizations of the hyper-singular integral operator on closed surfaces 
and show that the inverses of the corresponding system matrices can be approximated by blockwise
low-rank matrices at an exponential rate in the block rank. We cover in particular the data-space
format of $\H$-matrices. We show the approximability result for two types of discretizations. 
The first one is a saddle point formulation, which incorporates the constraint of 
vanishing mean of the solution. The second discretization is based on 
a stabilized hyper-singular operator, which leads to symmetric positive definite matrices. 
In this latter setting, we also show that the hierarchical Cholesky factorization can be 
approximated at an exponential rate in the block rank. 
\end{abstract}

\section{Introduction}
Boundary element method (BEM) are obtained as the discretizations of 
boundary boundary integral equations. These arise, for example, when 
elliptic partial differential equations are reformulated as integral equations on the 
boundary $\Gamma:= \partial\Omega$ of a domain $\Omega \subset \R^d$. A particular strength
of these methods is that they can deal with unbounded exterior domains. Reformulating 
an equation posed in a volume as one on its boundary brings about a significant reduction
in complexity. However, the boundary integral operators are fully occupied, and 
this has sparked the development of various matrix compression techniques. One possibility, 
which we will not pursue here,  
are wavelet compression techniques, 
\cite{rathsfeld98,rathsfeld01,schneider98,petersdorff-schwab-schneider97,tausch03,tausch-white03},
where sparsity of the system matrices results from the choice of basis. 
In the present work, we will consider data-sparse matrix formats that are based on blockwise 
low-rank matrices. These formats can be traced back to 
multipole expansions, \cite{rokhlin85,greengard-rokhlin97}, 
panel clustering, \cite{hackbusch-nowak88,hackbusch-nowak89,hackbusch-sauter93,sauter92},
and were then further developed in the mosaic-skeleton method, \cite{tyrtyshnikov00}, 
the adaptive cross approximation (ACA)
method, \cite{bebendorf00}, and the hybrid cross approximation (HCA), \cite{boerm-grasedyck05}. 
A fairly general framework for these techniques is given by 
the ${\mathcal H}$-matrices, introduced 
in \cite{Hackbusch99,GrasedyckHackbusch,GrasedyckDissertation,HackbuschBuch} 
and the $\H^2$-matrices, \cite{HackbuschKhoromskijSauter,Boerm,BoermBuch}. Both $\H$- and 
$\H^2$-matrices come with
an (approximate) arithmetic and thus provide the possibility of 
(approximately) inverting or factorizing a BEM matrix; also 
algebraic approaches to the design of preconditioners for boundary element discretizations, both 
for positive and negative order operators, are available with this framework.
Empirically, it has already been observed in \cite{GrasedyckDissertation,Bebendorf05} that such
an approach works well in practice.

Mathematically, the fundamental question in connection with the $\H$-matrix arithmetic is whether the desired
result, i.e., the inverse (or a factorization such as an $LU$- or Cholesky factorization), can 
be represented accurately in near optimal complexity in this format. This question is 
answered in the affirmative in the present work for discretizations of the hyper-singular integral
operator associated with the Laplace operator. In previous work, we showed similar existence results for 
FEM discretizations \cite{FMPFEM} and the discretization of the 
single layer operator, \cite{FMPBEM}. Compared to the symmetric positive definite case of the 
single layer operator studied in \cite{FMPBEM}, the hyper-singular operator on closed surfaces has 
a one-dimensional kernel and is naturally treated as a (simple) saddle point problem. 
We show
in Theorem~\ref{th:Happrox} (cf. also Remark~\ref{rem:saddle-point})
that the inverse of the discretization of this saddle point formulation can be approximated
by blockwise low-rank matrices at an exponential rate in the block rank. A corresponding 
approximation result for the discretized version of the stabilized hyper-singular operator 
follows then fairly easily in Corollary~\ref{cor:stabGalerkin}. The approximation result 
Theorem~\ref{th:Happrox} also underlies our proof that the hierarchical Cholesky factorization
of the stabilized hyper-singular operator admits an efficient representation 
in the $\H$-matrix format (Theorem~\ref{th:HLU}). 

The approximability problem for the inverses of Galerkin BEM-matrices has previously only
been studied in \cite{FMPBEM} for the single layer operator. In a FEM context, works prior 
to \cite{FMPFEM} include \cite{Bebendorf,Bebendorf4,Bebendorf05}, \cite{schreittmiller06}, and 
\cite{Boerm}. These works differ from \cite{FMPFEM,FMPBEM} and the present paper in an important
technical aspect: while \cite{FMPFEM,FMPBEM} and the present analysis analyze the discretized 
operators and show exponential convergence in the block rank, the above mentioned works 
study first low-rank approximations on the continuous level and transfer these to the discrete
level in a final projection step. Therefore, they achieve exponential convergence in the block rank 
up to this projection error, which is related to the discretization error. 

The paper is structured as follows. In the interest of readability, we have collected 
the main result concerning the approximability of the inverse of the discretization of the 
saddle point formulation  in Section~\ref{sec:main-results}.  The mathematical core is 
found in Section~\ref{sec:Approximation-solution}, where we study how well solutions of the 
(discretized) hyper-singular integral equation can be approximated from low-dimensional spaces
(Theorem~\ref{thm:function-approximationHypSing}). In contrast to \cite{FMPBEM}, which considered
only lowest-order discretization, we consider here arbitrary fixed-order discretizations. 
The approximation result of Section~\ref{sec:Approximation-solution} can be translated
to the matrix level, which is done in Section~\ref{sec:H-matrix-approximation}. 
Section~\ref{sec:stabGalerkin} shows how the results for the saddle point formulation imply 
corresponding ones for the stabilized hyper-singular operator. 
Finally, Section~\ref{sec:LU-decomposition} provides the existence of an 
approximate $\H$-Cholesky decomposition. We close with numerical examples in 
Section~\ref{sec:numerics}. 

\medskip

We use standard integer order Sobolev spaces and the fractional order Sobolev spaces $H^{1/2}(\Gamma)$ 
and its dual $H^{-1/2}(\Gamma)$ as defined in, e.g., \cite{SauterSchwab}. 
The notation $\lesssim$ abbreviates $\leq$ up to a 
constant $C>0$ that depends only on the domain $\Omega$, the spatial dimension $d$, 
the polynomial degree $p$, and the $\gamma$-shape regularity of $\mathcal{T}_h$. It does not, however, 
depend on critical parameters such as the mesh size $h$, the dimension of the finite dimensional BEM
space, or the block rank employed. 
Moreover, we use $\simeq$ to indicate that both estimates
$\lesssim$ and $\gtrsim$ hold.

\section{Main Result}
\label{sec:main-results}
\subsection{Notation and setting} 
Throughout this paper, we assume that $\Omega \subset \R^{d}$, $d \in \{2,3\}$ is a
bounded Lipschitz domain such that $\Gamma:=\partial \Omega$ 
is polygonal (for $d = 2$) or polyhedral (for $d = 3$).
We assume that $\Gamma$ is connected.

We consider the hyper-singular integral operator $W \in L(H^{1/2}(\Gamma),H^{-1/2}(\Gamma))$ given by 
\bee
W v(x) = -\gamma_{1,x}^{\text{int}}(\widetilde{K}v)(x) = -\gamma_{1,x}^{\text{int}}\int_{\Gamma}(\gamma_{1,y}^{\text{int}}G(x-y))v(y) ds_y, \quad x \in \Gamma,
\ee
where $G(x) = -\frac{1}{2\pi} \log\abs{x}$ for $d=2$ and $G(x) = \frac{1}{4\pi}\frac{1}{\abs{x}}$ for $d=3$ 
is the fundamental solution associated with the Laplacian. Here, the double layer potential 
$\widetilde{K} \in L(H^{1/2}(\Gamma),H^{1}(\Omega))$ 
is given by $\widetilde{K}v(x) := \int_{\Gamma}(\gamma_{1,y}^{\text{int}}G(x-y))v(y) ds_y$, 
where $\gamma_{1,z}^{\text{int}}$ denotes the interior conormal derivative at the point $z \in \Gamma$, i.e., 
with the normal vector $n(z)$ at $z \in \Gamma$ pointing into $\Omega^c$ and some sufficiently smooth function $u$
defined in  $\Omega$ one requires $\gamma_{1,z}^{\text{in}} u = \nabla u(z) \cdot n(z)$. 

The hyper-singular integral operator $W$ is symmetric, positive semidefinite on $H^{1/2}(\Gamma)$.
Since $\Gamma$ is connected, $W$ has a one-dimensional kernel given by the constant functions. 
In order to deal with this kernel, we can either use factor spaces, stabilize the operator, 
or study a saddle point formulation. 
In the following, we will employ the latter by adding the side constraint of vanishing mean. 
In Section~\ref{sec:stabGalerkin} we will very briefly study the case of the stabilized operator, and 
our analysis of Cholesky factorizations in Section~\ref{sec:LU-decomposition} will be performed 
for the stabilized operator. 

With the bilinear form $b(v,\mu) := \mu \int_{\Gamma}v ds_x$, 
we get the saddle point formulation of the boundary integral equation
\bee
W\phi = f \quad \text{on} \;\Gamma
\ee
with arbitrary $f \in H^{-1/2}(\Gamma)$ as finding $(\phi,\lambda) \in H^{1/2}(\Gamma) \times \R$ 
such that
\begin{subequations}\label{eq:modelHScont}
\begin{align} 
 \skp{W\phi,\psi} + b(\psi,\lambda) &= \skp{f,\psi} \qquad \forall \psi \in H^{1/2}(\Gamma), \\
b(\phi,\mu) &= 0 \qquad \forall \mu \in \R. 
\end{align}
\end{subequations}
By classical saddle-point theory, this problem has a unique solution $(\phi,\lambda) \in H^{1/2}(\Gamma) \times \R$,
since the bilinear form $b$ satisfies an inf-sup condition, and the bilinear form $\skp{W\phi,\psi}$ is coercive
on the kernel of $b(\cdot,\lambda)$, which is just the one-dimensional space of constant functions 
(see, e.g., \cite{SauterSchwab}).

For the discretization, we assume that $\Gamma$ is triangulated by a (globally) {\it quasiuniform} mesh 
${\mathcal T}_h=\{T_1,\dots,T_M\}$ of mesh width $h := \max_{T_j\in \mathcal{T}_h}{\rm diam}(T_j)$. 
The elements $T_j \in \mathcal{T}_h$ are open line segments ($d=2$) or triangles ($d=3$). 
Additionally, we assume that the mesh $\T_h$ is regular in the sense of Ciarlet and 
$\gamma$-shape regular in the sense that for $d=2$ the quotient of the diameters of neighboring elements 
is bounded by $\gamma$ and for $d=3$ we have 
${\rm diam}(T_j) \le \gamma\,|T_j|^{1/2}$ for all $T_j\in\mathcal{T}_h$, where $|T_j| = \operatorname*{area}(T_j)$
denotes the length/area of the element $T_j$. 

We consider the Galerkin discretization of $W$ by continuous, piecewise polynomial functions of 
fixed degree $p \geq 1$ in 
$S^{p,1}({\mathcal T}_h) := \{u \in C(\Gamma)\, :\, u|_T \in P_p(T) \, \forall T \in \T_h \}$, where
$P_p(T)$ denotes the space of polynomials of maximal degree $p$ on the triangle $T$.
We choose a basis of $S^{p,1}({\mathcal T}_h)$, which is denoted by
 ${\mathcal B}_h:= \{\psi_j\, :\, j = 1,\dots, N\}$. 
Given that our results are formulated for matrices, assumptions on the basis ${\mathcal B}_h$ 
need to be imposed. For the 
isomorphism $\Phi:\R^N\ra S^{p,1}({\mathcal T}_h)$, $\mathbf{x} \mapsto \sum_{j=1}^Nx_j\psi_j$, 
we require 
\been\label{eq:basisisomorphism}
h^{(d-1)/2}\norm{\mathbf{x}}_2 \lesssim \norm{\Phi(\mathbf{x})}_{L^2(\Gamma)} \lesssim h^{(d-1)/2}\norm{\mathbf{x}}_2
\quad \forall\, \mathbf{x} \in \R^N. 
\een

\begin{remark}
{\rm 
The standard basis for $p=1$ consists of the classical hat functions satisfying 
$\psi_j(x_i) = \delta_{ij}$ and for $p\geq 2$ 
we refer to, e.g., \cite{SchwabBuch,karniadakis-sherwin99,demkowicz-kurtz-pardo-paszynski-rachowicz-zdunek08}.
These bases satisfy assumption \eqref{eq:basisisomorphism}.}
\eremk
\end{remark}

The discrete variational problem is given by finding $(\phi_h,\lambda_h) \in S^{p,1}(\T_h) \times \R$ 
such that
\bean\label{eq:modelHS}
\skp{W\phi_h,\psi_h} + b(\psi_h,\lambda_h) &=& \skp{f,\psi_h} \qquad \forall \psi_h \in S^{p,1}(\T_h), \\
b(\phi_h,\mu) &=& 0 \qquad \forall \mu \in \R. \nonumber
\eean

Since the bilinear form $b$ trivially satisfies a discrete inf-sup condition, the discrete problem 
is uniquely solvable as well, and one has the stability bounds 
\begin{equation}
\label{eq:discrete-stability-1} 
\norm{\phi_h}_{H^{1/2}(\Gamma)} + |\lambda| \leq C \norm{f}_{H^{-1/2}(\Gamma)},  
\end{equation}
for a constant $C > 0$ which depends only on $\Gamma$. For $f \in L^2(\Gamma)$ and the $L^2$-projection 
$\Pi^{L^2}:L^2(\Gamma) \rightarrow S^{p,1}(\T_h)$, one even has the following estimate 
\begin{equation}
\label{eq:discrete-stability} 
\norm{\phi_h}_{H^{1/2}(\Gamma)} + |\lambda| \leq C \norm{\Pi^{L^2} f}_{L^{2}(\Gamma)} 
\leq C \norm{f}_{L^2(\Gamma)}. 
\end{equation}

With the basis $\mathcal{B}_h$, the left-hand side of \eqref{eq:modelHS} leads to the invertible block matrix
\been
\label{eq:Wtilde}
\boldsymbol{\mathcal{W}} := \bpm \mathbf{W} & \mathbf{B} \\ \mathbf{B}^T & 0 \epm,
\een
where the matrix $\mathbf{W} \in \R^{N\times N}$ and the vector $\mathbf{B} \in \R^{N\times 1}$ are given by
\been\label{eq:matrixHypsing}
\mathbf{W}_{jk} = \skp{W\psi_k,\psi_j}, \quad \mathbf{B}_j = \skp{\psi_j,1}, \quad \psi_k,\psi_j \in \mathcal{B}_h.
\een

\subsection{Approximation of $\boldsymbol{\mathcal{W}}^{-1}$ by blockwise low-rank matrices} 
Our goal is to approximate the inverse matrix $\boldsymbol{\mathcal{W}}^{-1}$ by $\H$-matrices, which 
are based on the concept that certain 'admissible' blocks can be approximated by low-rank factorizations.
The following definition specifies for which blocks such a factorization can be derived. 

\begin{definition}[bounding boxes and $\eta$-admissibility]
\label{def:admissibility}
A \emph{cluster} $\tau$ is a subset of the index set $\mathcal{I} = \{1,\ldots,N\}$. For a cluster $\tau \subset \mathcal{I}$, 
we say that $B_{R_{\tau}} \subset \R^d$ 
is a \emph{bounding box} if: 
\begin{enumerate}[(i)]
 \item 
\label{item:def:admissibility-i}
$B_{R_{\tau}}$ is a hyper cube with side length $R_{\tau}$,
 \item 
\label{item:def:admissibility-ii}
$ \supp \psi_i \subset B_{R_{\tau}}$ for all $ i \in \tau $.
\end{enumerate}

For an admissibility parameter $\eta > 0$, 
a pair of clusters $(\tau,\sigma)$ with $\tau,\sigma \subset \mathcal{I}$ 
is $\eta$-\emph{admissible} if 
there exist bounding boxes $B_{R_{\tau}}$, $B_{R_{\sigma}}$ satisfying 
(\ref{item:def:admissibility-i})--(\ref{item:def:admissibility-ii})
such that
\been\label{eq:admissibility}
\min\{{\rm diam}(B_{R_{\tau}}),{\rm diam}(B_{R_{\sigma}})\} \leq  \eta \; {\rm dist}(B_{R_{\tau}},B_{R_{\sigma}}).
\een
\end{definition}

\begin{definition}[blockwise rank-$r$ matrices]
Let $P$ be a partition of ${\mathcal I} \times {\mathcal I}$ and $\eta>0$. 
A matrix ${\mathbf W}_{\mathcal{H}} \in \mathbb{R}^{N \times N}$ 
is said to be a \emph{blockwise rank-$r$ matrix}, if for every $\eta$-admissible cluster pair $(\tau,\sigma) \in P$, 
the block ${\mathbf W}_{\mathcal{H}}|_{\tau \times \sigma}$ is a rank-$r$ matrix, i.e., it has the form 
${\mathbf W}_{\mathcal{H}}|_{\tau \times \sigma} = {\mathbf X}_{\tau \sigma} {\mathbf Y}^T_{\tau \sigma}$ with 
$\mathbf{X}_{\tau\sigma} \in \mathbb{R}^{\abs{\tau}\times r}$ 
and $\mathbf{Y}_{\tau\sigma} \in \mathbb{R}^{\abs{\sigma}\times r}$.
Here and below, $\abs{\sigma}$ denotes the cardinality of a finite set $\sigma$.
\end{definition}

\begin{definition}[cluster tree]
A \emph{cluster tree} with \emph{leaf size} $n_{\rm leaf} \in \mathbb{N}$ is a binary tree $\mathbb{T}_{\mathcal{I}}$ with root $\mathcal{I}$  
such that for each cluster $\tau \in \mathbb{T}_{\mathcal{I}}$ the following dichotomy holds: either $\tau$ is a leaf of the tree and 
$\abs{\tau} \leq n_{\rm leaf}$, or there exist sons $\tau'$, $\tau'' \in \mathbb{T}_{\mathcal{I}}$, which are disjoint subsets of $\tau$ with 
$\tau = \tau' \cup \tau''$. The \emph{level function} ${\rm level}: \mathbb{T}_{\mathcal{I}} \rightarrow \mathbb{N}_0$ is inductively defined by 
${\rm level}(\mathcal{I}) = 0$ and ${\rm level}(\tau') := {\rm level}(\tau) + 1$ for $\tau'$ a son of $\tau$. The \emph{depth} of a cluster tree
is ${\rm depth}(\mathbb{T}_{\mathcal{I}}) := \max_{\tau \in \mathbb{T}_{\mathcal{I}}}{\rm level}(\tau)$.  
\end{definition}

\begin{definition}[far field, near field, and sparsity constant]
 A partition $P$ of $\mathcal{I} \times \mathcal{I}$ is said to be based on the cluster tree $\mathbb{T}_{\mathcal{I}}$, 
if $P \subset \mathbb{T}_{\mathcal{I}}\times\mathbb{T}_{\mathcal{I}}$. For such a partition $P$ 
and a fixed admissibility parameter $\eta > 0$, we define the \emph{far field} and the \emph{near field} 
as 
\begin{equation}\label{eq:farfield}
P_{\rm far} := \{(\tau,\sigma) \in P \; : \; (\tau,\sigma) \; \text{is $\eta$-admissible}\}, \quad P_{\rm near} := P\setminus P_{\rm far}.
\end{equation}
The \emph{sparsity constant} $C_{\rm sp}$ of such a partition 
was introduced in \cite{GrasedyckDissertation} as 
\begin{equation}\label{eq:sparsityConstant}
C_{\rm sp} := \max\left\{\max_{\tau \in \mathbb{T}_{\mathcal{I}}}\abs{\{\sigma \in \mathbb{T}_{\mathcal{I}} \, : \, \tau \times \sigma \in P_{\rm far}\}},\max_{\sigma \in \mathbb{T}_{\mathcal{I}}}\abs{\{\tau \in \mathbb{T}_{\mathcal{I}} \, : \, \tau \times \sigma \in P_{\rm far}\}}\right\}.
\end{equation}
\end{definition}

The following theorem is the main result of this paper. It states that
the inverse matrix $\boldsymbol{\mathcal{W}}^{-1}$ can be approximated by an $\H$-matrix, 
where the approximation error in the spectral norm converges exponentially in the block rank.

\bthe\label{th:Happrox}
Fix an admissibility parameter $\eta >0$. Let a partition $P$ of $\mathcal{I}\times\mathcal{I}$ be based 
on the cluster tree $\mathbb{T}_{\mathcal{I}}$.
Then, there exists a blockwise rank-$r$ matrix $\mathbf{V}_{\H}$ such that

\bee
\norm{\boldsymbol{\mathcal{W}}^{-1}|_{N\times N} - \mathbf{V}_{\H}}_2 \leq C_{\rm apx}C_{\rm sp} 
{\rm depth}(\mathbb{T}_{\mathcal{I}}) 
N^{(2d-1)/(2d-2)} e^{-br^{1/(d+1)}}.
\ee

The constant $C_{\rm apx}$ depends only on $\Omega$, $d$, $p$, and the $\gamma$-shape 
regularity of the quasiuniform triangulation $\T_h$, while
the constant $b>0$ additionally depends on $\eta$.
\ethe

\begin{remark}[approximation of inverse of full system]
\label{rem:saddle-point}
{\rm 
The previous theorem provides an approximation $\mathbf{V}_{\H}$ to the first 
$N\times N$-subblock $\mathbf{V}$ of the matrix 
$\boldsymbol{\mathcal{W}}^{-1} = \bpm \mathbf{V} & \mathbf{P} \\ \mathbf{P}^T & 0\epm$. 
Since $\mathbf{P} \in \R^{N\times 1}$ is a vector, the matrix 
$\widehat{\mathbf{V}}_{\H} = \bpm \mathbf{V}_{\H}  & \mathbf{P} \\ \mathbf{P}^T & 0\epm$ is a 
blockwise rank-$r$ approximation to the matrix $\boldsymbol{\mathcal{W}}^{-1}$ satisfying

\bee
\norm{\boldsymbol{\mathcal{W}}^{-1} - \widehat{\mathbf{V}}_{\H}}_2 \leq 
C_{\rm apx}C_{\rm sp} {\rm depth}(\mathbb{T}_{\mathcal{I}}) 
N^{(2d-1)/(2d-2)} e^{-br^{1/(d+1)}}.
\ee
}\eremk
\end{remark}

\begin{remark}[relative errors]
{\rm 
In order to derive a bound for the relative error, we need an estimate on 
$\norm{\boldsymbol{\mathcal{W}}}_2$, since 
$\frac{1}{\norm{{\boldsymbol{\mathcal{W}}}^{-1}}_2} \leq \norm{\boldsymbol{\mathcal{W}}}_2$.  
Since $\mathbf{W}$ is symmetric it suffices to estimate the Rayleigh quotient.
The continuity of the hyper-singular integral operator as well as an inverse inequality, see 
Lemma~\ref{lem:inverseinequality} below, and \eqref{eq:basisisomorphism} imply
\begin{eqnarray*}
\skp{\boldsymbol{\mathcal{W}}\bpm \mathbf{v} \\ \lambda \epm,\bpm\mathbf{v} \\ \lambda \epm} 
&\lesssim& \norm{v}_{H^{1/2}(\Gamma)}^2 + \abs{\lambda \skp{v,1}} \\
&\lesssim& h^{-1}\norm{v}_{L^{2}(\Gamma)}^2 + \abs{\lambda}\norm{v}_{L^2(\Gamma)}
\lesssim h^{d-2}\norm{\bpm \mathbf{v} \\ \lambda \epm}_2^2. 
\end{eqnarray*}
Using $h \simeq N^{-1/(d-1)}$, we get a bound for the relative error

\begin{equation}
\frac{\norm{{\boldsymbol{\mathcal{W}}}^{-1} - \widehat{\mathbf{V}}_{\H}}_2}{\norm{\boldsymbol{\mathcal{W}}^{-1}}_2 }
\lesssim C_{\rm apx} C_{\rm sp} N^{(d+1)/(2d-2)} {\rm depth}(\mathbb{T}_{\mathcal{I}}) e^{-br^{1/(d+1)}}.
\end{equation}
}\eremk
\end{remark}

\section{Approximation of the potential}
\label{sec:Approximation-solution}
In order to approximate the inverse matrix $\boldsymbol{\mathcal{W}}^{-1}$ by a blockwise low-rank matrix, 
we will analyze how well the solution of \eqref{eq:modelHS} can be approximated from low dimensional spaces.

Solving the problem \eqref{eq:modelHS} is equivalent to solving the linear system
\been\label{eq:linearsystem}
\bpm \mathbf{W} & \mathbf{B} \\ \mathbf{B}^T & 0 \epm \bpm\mathbf{x} \\ \lambda \epm = \bpm \mathbf{b} \\0\epm
\een
with $\mathbf{W}$, $\mathbf{B}$ from \eqref{eq:matrixHypsing} and $\mathbf{b} \in \R^N$ 
defined by $\mathbf b_j = \skp{f,\psi_j}$.

The solution vector $\mathbf{x}$ is linked to the Galerkin solution $\phi_h$ from \eqref{eq:modelHS} via 
$\phi_h = \sum_{j=1}^N \mathbf{x}_j\psi_j$.

In this section, we will repeatedly use the $L^2(\Gamma)$-orthogonal projection
$\Pi^{L^2}:L^2(\Gamma)\ra S^{p,1}(\T_h)$ onto $S^{p,1}(\T_h)$, which, we recall, is defined by 
\been\label{L2projection}
\skp{\Pi^{L^2}v,\psi_h} = \skp{v,\psi_h} \quad \forall \psi_h \in S^{p,1}(\T_h).
\een

The following theorem is the main result of this section; it states that
for an admissible block $(\tau,\sigma)$, there exists a low dimensional approximation space such that the 
restriction to $B_{R_{\tau}}\cap\Gamma$ of the Galerkin solution $\phi_h$ can be approximated well  
from it as soon as the right-hand side $f$ has support in $B_{R_{\sigma}}\cap\Gamma$.

\begin{theorem}\label{thm:function-approximationHypSing}
Let $(\tau,\sigma)$ be a cluster pair with bounding boxes $B_{R_\tau}$, $B_{R_\sigma}$
(cf. Definition~\ref{def:admissibility}).
Assume $ \eta\, {\rm dist}(B_{R_\tau},B_{R_\sigma}) \geq {\rm diam}(B_{R_\tau})$ 
for some admissibility parameter $\eta > 0$. 
Fix $q \in (0,1)$. 
Then, for each $k\in\mathbb{N}$ there exists a space
$W_k\subset S^{p,1}({\mathcal T}_h)$ with $\dim W_k\leq C_{\rm dim} (2+\eta)^d q^{-d}k^{d+1}$ 
such that for arbitrary $f \in L^2(\Gamma)$ with
$\supp  f  \subset B_{R_\sigma}\cap\Gamma$, the solution $\phi_h$ of \eqref{eq:modelHS} 
satisfies 
\begin{equation}
\label{eq:thm:function-approximation-1HS}
\min_{w \in W_k} \|\phi_h - w\|_{L^2(B_{R_{\tau}}\cap\Gamma)} 
\leq C_{\rm box}  h^{-1/2} q^k \|\Pi^{L^2} f\|_{L^2(\Gamma)}
\leq C_{\rm box}  h^{-1/2} q^k \|f\|_{L^2(\Gamma)}.
\end{equation}
The constants $C_{\rm dim}$,  $C_{\rm box}>0$ depend only on $\Omega$, $d$, $p$, 
and the $\gamma$-shape regularity of the quasiuniform triangulation $\mathcal{T}_h$.
\end{theorem}

The proof of Theorem~\ref{thm:function-approximationHypSing} will be given at the end 
of this section. Its main ingredients can be summarized as follows:
First, the double-layer potential
$$
u(x) := \widetilde{K}\phi_h(x) = \int_{\Gamma}\gamma_{1,y}^{\text{int}} G(x-y)\phi_h(y)ds_y, 
\qquad x \in \R^d \setminus \Gamma,
$$ 
generated by the solution $\phi_h$ of \eqref{eq:modelHS} 
is harmonic on $\Omega$ as well as on $\Omega^c := \R^d \setminus \overline{\Omega}$ and 
satisfies the jump conditions 
\bean\label{eq:jumpconditions}
[\gamma_0 u] &:=& \gamma_0^{\text{ext}}u - \gamma_0^{\text{int}}u = \phi_h \in H^{1/2}(\Gamma),\nonumber \\
\unjump &:=& \gamma_1^{\text{ext}}u - \gamma_1^{\text{int}}u = 0 \in H^{-1/2}(\Gamma).
\eean
Here, $\gamma_0^{\text{ext}},\gamma_0^{\text{int}}$ denote the exterior and interior trace operator and 
$\gamma_1^{\text{ext}},\gamma_1^{\text{int}}$ the exterior and interior conormal derivative, 
see, e.g., \cite{SauterSchwab}.
Hence, the potential $u$ is in a space of piecewise harmonic functions, where 
the jump across the boundary is a continuous piecewise polynomial of degree $p$,
and the jump of the normal derivative vanishes. 
These properties will characterize the spaces ${\mathcal H}_h(D)$ to be introduced below. 
The second observation is an orthogonality condition on admissible blocks $(\tau,\sigma)$. 
For right-hand sides $f$ with $\supp f  \subset B_{R_{\sigma}}\cap \Gamma$, 
equation \eqref{eq:modelHS}, the admissibility condition, and $W = - \gamma_1^{\text{int}}\widetilde{K}$ imply 
\begin{equation}
\label{eq:orthogonalityHS}
-\skp{\gamma_1^{\text{int}} u,\psi_{h}} + \lambda\skp{\psi_{h},1} = 0
\quad \forall \psi_{h} \in S^{p,1}({\mathcal T}_{h}) \, \text{with} \, \supp\psi_{h} \subset B_{R_{\tau}}\cap\Gamma. 
\end{equation}
For a cluster $\rho \subset \mathcal{I}$, we define $\Gamma_{\rho} \subset \Gamma$ as an open polygonal
manifold given by 
\been\label{eq:screen}\Gamma_{\rho} := {\rm interior} \left(\bigcup_{j\in\rho}\supp \psi_{j}\right).\een

Let $D$ be an open set and $D^- := D\cap\Omega$, $D^+ := D\cap\overline{\Omega}^c$. A function
$v \in H^1(D\setminus\Gamma)$
is called piecewise harmonic, if
$$
\int_{D\setminus\Gamma}\nabla v \cdot \nabla \varphi\, dx = 0 \quad \forall\varphi \in C_0^{\infty}(D^{\pm}).
$$

\begin{definition}

Let $D \subset \R^{d}$ be open. The restrictions 
of the interior and exterior trace operators $\gamma_0^{\text{int}}$, $\gamma_0^{\text{ext}}$ to 
$D \cap \Gamma$ are operators $\gamma_0^{\text{int}}|_{D \cap \Gamma}:H^1(D^-) \rightarrow L^2_{loc} (D \cap \Gamma)$ 
and $\gamma_0^{\text{ext}}|_{D \cap \Gamma}:H^1(D^+) \rightarrow L^2_{loc} (D \cap \Gamma)$ 
defined in the following way: 
For any (relative) compact $U \subset D \cap \Gamma$, one selects a cut-off function $\eta \in C^\infty_0(D)$ 
with $\eta \equiv 1$ on $U$. 
Since $u \in H^1(D^-)$ implies $\eta u \in H^1(\Omega)$, we have 
$\gamma_0^{\text{int}} \eta u \in H^{1/2}(\Gamma)$ and thus its restriction to $U$ is a well-defined 
function in $L^2(U)$. It is easy to see that the values on $U$ do not depend on the choice of $\eta$. 
The operator $\gamma_0^{\text{ext}} |_{D\cap\Gamma}$ is defined 
completely analogously. 


In order to define the restriction of the normal derivative of a piecewise harmonic function
$v \in H^1(D\setminus\Gamma)$, let $\eta \in C^{\infty}(\R^d)$
with $\supp \eta \subset D$ and $\eta \equiv 1$ on a compact set
$U\subset D$. Then, the exterior normal derivative $\partial_n(\eta v)$ is well defined
as a functional in $H^{-1/2}(\Gamma)$, and 
we define $\partial_n v|_{U}$ as the functional
\bee
\skp{\partial_n v|_{U},\varphi} = \skp{\partial_n(\eta v),\varphi}, \quad \forall \varphi \in H^{1/2}(\Gamma), 
\supp \varphi \subset U.
\ee
Again, this definition does not depend on the choice of $\eta$ as long as $\eta \equiv 1$ on $U$. 
\end{definition}

\begin{definition}
For a piecewise harmonic function $v \in H^1(D\setminus\Gamma)$, 
we define the jump of the normal derivative $\vnjump|_{D\cap\Gamma}$ on $D\cap\Gamma$ as the functional
\been\label{eq:unjumpdef}
\skp{\vnjump|_{D\cap\Gamma},\varphi} := \int_{D^+\cup D^-}\nabla v \cdot \nabla \varphi\, dx \quad \forall \varphi \in H^1_0(D). 
\een
\end{definition}
We note that the value $\skp{\vnjump|_{D\cap\Gamma},\varphi}$ depends only on $\varphi|_{D\cap \Gamma}$ 
in the sense that $\skp{\vnjump|_{D\cap\Gamma},\varphi} = 0$ for all 
$\varphi \in C_0^{\infty}(D)$ with $\varphi|_{D\cap\Gamma} = 0$.
Moreover, if $\vnjump|_{D\cap\Gamma}$ is a function in $L^2(D\cap\Gamma)$, then it is unique.
The definition \eqref{eq:unjumpdef} is consistent with \eqref{eq:jumpconditions} in the following sense:
For a potential $\widetilde{K}\phi_h$ with $\phi_h \in S^{p,1}(\T_h)$, we have the jump condition
$[\partial_n \widetilde{K} \phi_h]|_{D\cap\Gamma} = 0$.

With these observations, we can define the space
\begin{eqnarray*}
\mathcal{H}_h(D) &:=& \{v\in H^1(D\setminus\Gamma) \colon 
\text{$v$ is piecewise harmonic}, [\partial_n v]|_{D\cap\Gamma} = 0, \\
& &\phantom{v\in H^1(D\setminus\Gamma) \colon \;\,} \exists \widetilde{v} \in S^{p,1}({\mathcal T}_h) \;
\mbox{s.t.} \ 
\vjump|_{D\cap\Gamma} = \widetilde{v}|_{D\cap\Gamma}\}.
\end{eqnarray*}
The potential $u = \widetilde{K}\phi_h$ indeed satisfies $u \in \H_h(D)$ for any domain $D$; 
we will later take $D$ to be a box $B_R$. 


For a box $B_R$ with side length $R$, we introduce the following norm on $H^1(B_R\setminus\Gamma)$
\begin{equation*}
\triplenorm{v}_{h,R}^2 := \left(\frac{h}{R}\right)^2 \norm{\nabla v}^2_{L^2(B_R\setminus\Gamma)} + 
\frac{1}{R^2}\norm{v}_{L^2(B_R\setminus\Gamma)}^2,
\end{equation*}
which is, for fixed $h$, equivalent to the $H^1(B_R\setminus\Gamma)$-norm. \\

A main tool in our proofs is the nodal interpolation operator 
$I_h: C(\Gamma) \rightarrow S^{p,1}({\mathcal T}_h)$. 
Since $p+1>\frac{d-1}{2}$ (recall: $d \in \{2,3\}$), the interpolation operator $I_h$ has
the following local approximation property for continuous, $\mathcal{T}_h$-piecewise $H^{p+1}$-functions 
$u \in C(\Gamma)\cap H^{p+1}_{\text{pw}}(\mathcal{T}_h) := 
\{u\in C(\Gamma) : u|_{T} \in H^{p+1}(T)\, \forall \, T \in \mathcal{T}_h\}$:
\begin{equation}\label{eq:NIapprox}
\norm{u-I_h u}_{H^m(T)}^2 \leq C h^{2(p+1-m)} \left\vert{u}\right\vert_{H^{p+1}(T)}^2, 
\; 0 \leq m \leq p+1.
\end{equation}
The constant $C>0$ depends only on $\gamma$-shape regularity of the quasiuniform triangulation $\mathcal{T}_h$, 
the dimension $d$, and the polynomial degree $p$. 

In the following, we will approximate the Galerkin solution on 
certain nested boxes, which are concentric according to the following definition.
\begin{definition}
Two (open) boxes $B_R$, $B_{R^\prime}$
are said to be concentric boxes with side lengths $R$ and $R^\prime$, if 
they have the same barycenter and $B_R$ can be obtained by a stretching
of $B_{R^\prime}$ by the factor $R/R^\prime$ taking their common barycenter
as the origin. 
\end{definition}

The following lemma states two classical inverse inequalities for functions in $\mathcal{S}^{p,1}(\T_h)$,
which will repeatedly be used in this section. For a proof we refer to {\cite[Theorem 3.2]{GHS}}
and \cite[Theorem 4.4.2]{SauterSchwab}.

\blem\label{lem:inverseinequality}
There is a constant $C > 0$ depending only on $\Omega,d,p$, and the $\gamma$-shape 
regularity of the quasiuniform triangulation $\T_h$ such that for all $s \in [0,1]$ the inverse inequality
\been\label{eq:inverse}
\norm{v}_{H^{s}(\Gamma)}\leq C h^{-1/2}\norm{v}_{L^2(\Gamma)} \quad \forall v \in S^{p,1}(\T_h) 
\een
holds.
Furthermore, for $0\leq m\leq \ell$ the inverse estimate 
\been\label{eq:inverse2}
\norm{v}_{H^{\ell}(T)} \leq C h^{m-\ell} \norm{v}_{H^{m}(T)}, \quad \forall v \in \mathcal{P}^{p}(T)
\een
holds for all $T \in \T_h$, where the constant
$C>0$ depends only on $\Omega,p,\ell$ and the $\gamma$-shape 
regularity of the quasiuniform triangulation $\T_h$.
\elem
The following lemma shows that for piecewise harmonic functions, the restriction of the 
normal derivative is a function in $L^2$ on a smaller box, and provides an estimate 
of the $L^2$-norm of the normal derivative.
\blem 
\label{lem:estimateun}
Let $\delta \in (0,1)$, $R \in(0,2\operatorname{diam}(\Omega))$ be such that 
$\frac{h}{R}\leq \frac{\delta}{4}$, and let $\mu \in \R$.
Let $B_R$, $B_{(1+\delta)R}$ be two concentric boxes of 
side lengths $R$ and $(1+\delta)R$. 
Then, there exists a constant $C> 0$ depending only on $\Omega$, $d,p$, and the $\gamma$-shape regularity
of the quasiuniform triangulation $\T_h$, such that for 
all $v \in \H_{h}(B_{(1+\delta)R})$ we have
\begin{eqnarray}
\label{eq:lem:estimateunjump-ii}
\norm{\partial_n v}_{L^2(B_{R}\cap \Gamma)} \
\leq C h^{-1/2}\left(\norm{\nabla v}_{L^2(B_{(1+\delta)R}\setminus\Gamma)}+
\frac{1}{\delta R}\norm{v}_{L^2(B_{(1+\delta)R}\setminus\Gamma)}\right).
\end{eqnarray}
\elem
\bpro 

{\em 1.~step:} 
Let $\eta\in W^{1,\infty}({\mathbb R}^d)$ satisfy $0\leq\eta\leq 1$, $\eta \equiv 1$ on 
$B_{(1+\delta/2)R}$, $\supp \eta \subset B_{(1+\delta)R}$,
and $\norm{\nabla \eta}_{L^{\infty}(B_{(1+\delta)R})} \lesssim \frac{1}{\delta R}$.
In order to shorten the proof, we assume
$\gamma_0^{\rm int} \eta = \gamma_0^{\rm ext} \eta \in S^{1,1}(\mathcal{T}_h)$ so that inverse inequalities
are applicable. We mention in passing that this simplification could be avoided by using 
``super-approximation'', a technique that goes back to \cite{nitsche-schatz74} 
(cf., e.g., \cite[Assumption~{7.1}]{wahlbin91}). Let us briefly indicate, how the assumption 
$\eta \in S^{1,1}(\T_h)$ can be ensured: Start from a smooth cut-off function 
$\widetilde \eta \in C^\infty_0({\mathbb R}^d)$ with the desired support properties. Then, the 
piecewise linear interpolant $I^1_h \widetilde\eta \in S^{1,1}(\T_h)$ has the 
desired properties on $\Gamma$. It therefore suffices to construct a suitable lifting. 
This is achieved with the lifting operator described in \cite[Chap.~{VI}, Thm.~3]{stein70} and afterwards
a multiplication by a suitable cut-off function again.

{\em 2.~step:}
Let $z := \widetilde{K}(\gamma_0^{\rm int}\eta [v])$. Then with the jump conditions
\bee
[\partial_n z] = 0, \quad [z] = \gamma_0^{\rm int}\eta [v]
\ee
and the fact that $v$ is piecewise harmonic, we get that the function $v-z$ is harmonic 
in the box $B_{(1+\delta/2)R}$. Thus, the function $w:= \nabla(v-z)$ is harmonic in $B_{(1+\delta/2)R}$ as well.
It therefore satisfies the interior regularity (Caccioppoli) estimate
\been\label{eq:intreg}
\norm{\nabla w}_{L^2(B_{(1+\delta/4)R}\setminus\Gamma)} \lesssim \frac{1}{\delta R} 
\norm{w}_{L^2(B_{(1+\delta/2)R}\setminus\Gamma)};
\een
a short proof of this Caccioppoli inequality can be found, for example, in \cite{Bebendorf}.

We will need a second smooth cut-off function $\widetilde{\eta}$ with $0\leq\widetilde{\eta}\leq 1$,
$\widetilde{\eta} \equiv 1$ on 
$B_{R}$, and $\supp \widetilde{\eta}   \subset B_{(1+\delta/4)R}$ and 
$\norm{\nabla \widetilde{\eta}}_{L^{\infty}(B_{(1+\delta/4)R})} \lesssim \frac{1}{\delta R}$.
The multiplicative trace inequality, see, e.g., \cite{BrennerScott}, 
implies together with \eqref{eq:intreg}
and $\delta R \leq 2 \diam(\Omega)$ due to the assumptions on $\delta,R$ that
\bea
\norm{\widetilde{\eta}w}_{L^2(B_R \cap \Gamma)}^2 &\lesssim& 
\norm{\widetilde{\eta}w}_{L^2(B_{(1+\delta/4)R}\setminus\Gamma)}^2
+ \norm{\widetilde{\eta}w}_{L^2(B_{(1+\delta/4)R}\setminus\Gamma)}
\norm{\nabla(\widetilde{\eta}w)}_{L^2(B_{(1+\delta/4)R}\setminus\Gamma)} \\
&\lesssim& \norm{\widetilde{\eta}w}_{L^2(B_{(1+\delta/4)R}\setminus\Gamma)}^2
+ \norm{\widetilde{\eta}w}_{L^2(B_{(1+\delta/4)R}\setminus\Gamma)}
\left(\frac{1}{\delta R}\norm{w}_{L^2(B_{(1+\delta/4)R}\setminus\Gamma)}
+\norm{\nabla w}_{L^2(B_{(1+\delta/4)R}\setminus\Gamma)}\right) \\
&\lesssim& \frac{1}{\delta R}\norm{w}_{L^2(B_{(1+\delta/2)R}\setminus\Gamma)}^2.
\eea
Therefore and with $\partial_n z = \partial_n\widetilde{K}(\gamma_0^{\rm int}\eta [v])
= -W(\gamma_0^{\rm int}\eta [v])$, we can estimate the normal derivative of $v$ by
\bea
\norm{\partial_n v}_{L^2(B_R\cap\Gamma)} &\leq& \norm{w\cdot n}_{L^2(B_R\cap\Gamma)} + 
\norm{\partial_n z}_{L^2(B_R\cap\Gamma)} \\
&\lesssim& \frac{1}{\sqrt{\delta R}}\norm{w}_{L^2(B_{(1+\delta/2)R}\setminus\Gamma)} + 
\norm{W(\gamma_0^{\rm int}\eta[v])}_{L^2(B_R\cap\Gamma)}.
\eea
Since the hyper-singular integral operator is a continuous mapping from $H^{1}(\Gamma)$ to $L^2(\Gamma)$ and the
double layer potential is continuous from $H^{1/2}(\Gamma)$ to $H^{1}(\Omega)$ (see, e.g., 
\cite[Remark 3.1.18.]{SauterSchwab}), we get
with $h<\delta R$, the inverse inequality \eqref{eq:inverse} (note that 
$(\gamma_0^{\rm int}\eta) [v]$ is a piecewise polynomial), and the trace inequality 
\bea
\norm{\partial_n v}_{L^2(B_R\cap\Gamma)} &\lesssim& 
\frac{1}{\sqrt{\delta R}}\norm{w}_{L^2(B_{(1+\delta/2)R}\setminus\Gamma)} + 
\norm{(\gamma_0^{\rm int}\eta)[v]}_{H^1(\Gamma)} \\
&\lesssim& \frac{1}{\sqrt{\delta R}}\left(\norm{\nabla v}_{L^2(B_{(1+\delta/2)R}\setminus\Gamma)}+
\norm{\nabla z}_{L^2(B_{(1+\delta/2)R}\setminus\Gamma)}\right)
 +h^{-1/2}\norm{(\gamma_0^{\rm int}\eta)[v]}_{H^{1/2}(\Gamma)} \\
&\lesssim& h^{-1/2}\left(\norm{\nabla v}_{L^2(B_{(1+\delta/2)R}\setminus\Gamma)}+
\norm{(\gamma_0^{\rm int}\eta)[v]}_{H^{1/2}(\Gamma)}\right) \\
&\lesssim& h^{-1/2}\left(\norm{\nabla v}_{L^2(B_{(1+\delta/2)R}\setminus\Gamma)}+
\norm{\eta v}_{H^{1}(B_{(1+\delta)R}\setminus\Gamma)}\right) \\
&\lesssim& h^{-1/2}\left(\norm{\nabla v}_{L^2(B_{(1+\delta)R}\setminus\Gamma)}+
\frac{1}{\delta R}\norm{v}_{L^{2}(B_{(1+\delta)R}\setminus\Gamma)}\right), \\
\eea
which finishes the proof.
\epro

The previous lemma implies that for functions in $\H_h(B_{(1+\delta)R})$, the normal derivative 
is a function in $L^2(B_R\cap\Gamma)$. 
Together with the orthogonality properties that we have identified 
in \eqref{eq:orthogonalityHS}, this is captured by the following affine space 
$\H_{h,0}(D,\Gamma_{\rho},\mu)$: 
\begin{eqnarray}
\H_{h,0}(D,\Gamma_{\rho},\mu)&:=& {\mathcal H}_h(D) \cap 
\{v \in H^1(D\setminus\Gamma)\colon \supp\vjump|_{D\cap\Gamma} \subset \overline{\Gamma_{\rho}}, \\
& & \phantom{{\mathcal H}^1_h(D) \cap \{} \langle \partial_n v|_{D\cap\Gamma},\psi_h\rangle -
\mu\skp{\psi_h,1} = 0
\, \forall \psi_h \in S^{p,1}({\mathcal T}_h) \, \text{with} \, 
\supp \psi_h \subset D\cap \overline{\Gamma_{\rho}}\}. \nonumber
\end{eqnarray}

\blem\label{lem:closedsubspaceHS}
The spaces $\H_h(D)$ and $\H_{h,0}(D,\Gamma_{\rho},\mu)$ are closed subspaces of $H^1(D\setminus \Gamma)$.
\elem
\bpro
Let $(v^j)_{j\in \N} \subset \H_h(D)$ be a sequence converging to $v \in H^1(D\setminus \Gamma)$. 
With the definition of the jump $[\gamma_0 v^j]|_{D\cap\Gamma}$ and the continuity of the trace operator from
$H^1(\Omega)$ to $L^2(\Gamma)$, we get that the sequence $[\gamma_0 v^j]|_{D\cap\Gamma}$ converges in 
$L^2_{\rm loc}(D\cap\Gamma)$ to $[\gamma_0 v]|_{D\cap\Gamma}$, 
and since $S^{p,1}(\T_h)$ is finite dimensional, we get that 
$[\gamma_0 v]|_{D\cap\Gamma} = \widetilde{v}|_{D\cap\Gamma}$ with a function $\widetilde{v} \in S^{p,1}(\T_h)$.

Moreover, for $\varphi \in C_0^{\infty}(D^{\pm})$ we have
\bee
\skp{\nabla v,\nabla \varphi}_{L^2(D\setminus\Gamma)} = 
\lim_{j\ra \infty} \skp{\nabla v^j, \nabla \varphi}_{L^2(D\setminus\Gamma)} = 0,
\ee
so $v$ is piecewise harmonic on $D\setminus\Gamma$. By definition \eqref{eq:unjumpdef} and the same argument,
we get
$[\partial_n v]|_{D\cap\Gamma} = 0$, and therefore $\H_h(D)$ is closed.
The space $\H_{h,0}(D,\Gamma_{\rho},\mu)$ is closed, since the intersection of closed
spaces is closed.
\epro

A key ingredient of the proof of Theorem~\ref{thm:function-approximationHypSing} is
a Caccioppoli-type interior regularity estimate, which is proved by use of the orthogonality 
property \eqref{eq:orthogonalityHS}. 

\blem\label{lem:CaccioppoliHS}
Let $\delta \in (0,1)$, $R\in(0,2\diam(\Omega))$ such that 
$\frac{h}{R} \leq \frac{\delta}{8}$ 
and let $\Gamma_{\rho}\subset \Gamma$ be of the form \eqref{eq:screen}.
Let $B_R$, $B_{(1+\delta)R}$ be two concentric boxes and let $\mu \in \R$.
Then, there exists a constant $C > 0$ depending only on 
$\Omega$, $d$, $p$, and the $\gamma$-shape regularity of the quasiuniform triangulation $\T_h$ such that for all 
$v \in \H_{h,0}(B_{(1+\delta)R},\Gamma_{\rho},\mu)$ 
\been\label{eq:caccioppoliHS}
\norm{\nabla v}_{L^2(B_{R}\setminus \Gamma)} \leq C\left(\frac{1+\delta}{\delta}  \triplenorm{v}_{h,(1+\delta)R}
+((1+\delta)R)^{(d-1)/2}\abs{\mu}\right).
\een
\elem

\bpro
Let $\eta \in H^1(\R^d)$ be a cut-off function with $\supp \eta \subset B_{(1+\delta/2)R}$, 
$\eta \equiv 1$ on $B_{R}$, and $\norm{\nabla \eta}_{L^{\infty}(B_{(1+\delta)R})} \lesssim \frac{1}{\delta R}$. 
As in the proof of Lemma~\ref{lem:estimateun}, we may additionally assume that
$\gamma_0^{\rm int} \eta = \gamma_0^{\rm ext} \eta$ is a piecewise polynomial of degree 1 on 
each connected component of $\Gamma\cap B_{(1+\delta)R}$.
Since $h$ is the maximal element diameter, $8h \leq \delta R$ implies $T \subset B_{(1+\delta)R}$ 
for all $T \in \T_h$ with $T \cap \supp \eta  \neq \emptyset$.
Because $v$ is piecewise harmonic and $[\partial_n v]|_{B_{(1+\delta)R}\cap\Gamma}=0$, we get
\bean\label{eq:CaccHS1}
\norm{\nabla(\eta v)}_{L^2(B_{(1+\delta)R}\setminus \Gamma)}^2 &=& 
\int_{B_{(1+\delta)R}\setminus\Gamma}\nabla v \cdot \nabla(\eta^2 v)+v^2 \abs{\nabla \eta}^2 dx \nonumber  \\
&=&\langle\partial_n v, \eta^2 \vjump\rangle + \int_{B_{(1+\delta)R}\setminus\Gamma}{v^2\abs{\nabla \eta}^2  dx}. 
\eean
We first focus on the surface integral. 
With the nodal interpolation operator $I_h$ from \eqref{eq:NIapprox} 
and the orthogonality \eqref{eq:orthogonalityHS},
we get
\begin{eqnarray}
\langle\partial_n v, \eta^2 \vjump\rangle &=& 
\langle \partial_n v, \eta^2 \vjump - I_h (\eta^2 \vjump)\rangle + 
\mu\skp{I_h(\eta^2 \vjump),1}.
\label{eq:lem:Caccioppoli-10HS}
\end{eqnarray}
The approximation property \eqref{eq:NIapprox} leads to
\been\label{eq:cacctemp1}
\norm{\eta^2\vjump - I_h(\eta^2\vjump)}_{L^2(\Gamma)}^2 \lesssim h^{2(p+1)}
\sum_{T \in \T_h}\abs{\eta^2\vjump}_{H^{p+1}(T)}^2.
\een
Since for each $T \in \T_h$ we have  $\vjump|_T \in {\mathcal P}_p$,  
we get $D^{k}\vjump|_T = 0$ for all multiindices 
$k \in \N_0^{d}$ with $\abs{k}:=\sum_{i=1}^d k_i = p+1$ and 
$\eta|_T \in \mathcal{P}_1$ implies $D^j \eta|_T = 0$ for $j \in \N_0^{d}$ with $\abs{j} \geq 2$.
With the Leibniz product rule, a direct calculation (see \cite[Lemma 2]{FMPFEM} for details) leads to
\begin{eqnarray*}
\abs{\eta^2 \vjump}^2_{H^{p+1}(T)}  &\lesssim& 
\frac{1}{(\delta R)^2} \abs{\eta \vjump}_{H^{p}(T)}^2+\frac{1}{(\delta R)^4} \abs{\vjump}_{H^{p-1}(T)}^2,
\end{eqnarray*} 
where the suppressed constant depends on $p$. 
The inverse inequalities \eqref{eq:inverse2} given in 
Lemma~\ref{lem:inverseinequality} imply
\bean\label{eq:NIjumpest}
\norm{\eta^2\vjump - I_h(\eta^2\vjump)}_{L^2(\Gamma)}^2 
&\lesssim& h^{2(p+1)}\sum_{T \in \T_h}\left(
\frac{1}{(\delta R)^2} \abs{\eta \vjump}_{H^{p}(T)}^2+\frac{1}{(\delta R)^4} \abs{\vjump}_{H^{p-1}(T)}^2\right)\nonumber\\
&\lesssim& \frac{h^{3}}{(\delta R)^2} 
\norm{\eta\vjump}_{H^{1/2}(\Gamma)}^2 
+ \frac{h^{4}}{(\delta R)^4} \norm{\eta\vjump}_{L^{2}(B_{(1+\delta)R}\cap\Gamma)}^2.
\eean
With the trace inequality, we obtain
\bean\label{eq:traceestimate}
\norm{\eta \vjump}_{H^{1/2}(\Gamma)}^2 &=& 
\norm{\gamma_0^{\text{ext}}(\eta v) - \gamma_0^{\text{int}}(\eta v)}_{H^{1/2}(\Gamma)}^2 \nonumber \\
 &\lesssim& \norm{\eta v}_{L^2(\Omega)}^2 + \norm{\nabla(\eta v)}_{L^2(\Omega)}^2 + 
\norm{\eta v}_{L^2(\Omega^c)}^2 +\norm{\nabla(\eta v)}_{L^2(\Omega^c)}^2 \nonumber \\
&\leq&\norm{v}_{L^2(B_{(1+\delta)R}\setminus\Gamma)}^2 + 
\norm{\nabla(\eta v)}_{L^2(B_{(1+\delta)R}\setminus\Gamma)}^2.
\eean
In the same way, the multiplicative trace inequality implies
\bean\label{eq:tempmultrace}
\norm{\eta \vjump}_{L^{2}(\Gamma)}^2 \lesssim \frac{1}{\delta R}\norm{\eta v}_{L^2(B_{(1+\delta)R}\setminus\Gamma)}^2 + 
\norm{\eta v}_{L^2(B_{(1+\delta)R}\setminus\Gamma)}\norm{\eta \nabla v}_{L^2(B_{(1+\delta)R}\setminus\Gamma)}.
\eean
We apply Lemma~\ref{lem:estimateun} with $\widetilde{R}=(1+\delta/2)R$ and 
$\widetilde{\delta}=\frac{\delta}{2+\delta}$ such that $(1+\widetilde{\delta})\widetilde{R}=(1+\delta)R$.
Together with \eqref{eq:NIjumpest} -- \eqref{eq:tempmultrace}, we get 
\begin{align*}
&\abs{\langle \partial_n v, \eta^2 \vjump - I_h (\eta^2 \vjump)\rangle} \leq 
\norm{\partial_n v}_{L^{2}(B_{(1+\delta/2)R}\cap\Gamma)}
\norm{\eta^2\vjump - I_h(\eta^2\vjump)}_{L^2(\Gamma)} \\
&\qquad\leq C \left(\norm{\nabla v}_{L^2(B_{(1+\delta) R}\setminus\Gamma)}+
\frac{1}{\delta R}\norm{v}_{L^2(B_{(1+\delta) R}\setminus\Gamma)}\right)
\bigg\{ \frac{h}{\delta R}\left(\norm{v}_{L^2(B_{(1+\delta)R}\setminus\Gamma)}+ 
\norm{\nabla(\eta v)}_{L^2(B_{(1+\delta)R}\setminus\Gamma)}\right) \\
&\qquad\qquad +
 \frac{h^{3/2}}{(\delta R)^2}\left(\frac{1}{(\delta R)^{1/2}}\norm{v}_{L^2(B_{(1+\delta)R}\setminus\Gamma)}+ 
\norm{\eta v}^{1/2}_{L^2(B_{(1+\delta)R}\setminus\Gamma)}\norm{\eta\nabla v}^{1/2}_{L^2(B_{(1+\delta)R}\setminus\Gamma)}\right)
\bigg\} \\
&\qquad\leq C\frac{h^2}{(\delta R)^2}\norm{\nabla v}_{L^2(B_{(1+\delta)R}\setminus \Gamma)}^2 + 
C\frac{1}{(\delta R)^2}\norm{v}_{L^2(B_{(1+\delta)R}\setminus \Gamma)}^2+ 
\frac{1}{4} \norm{\nabla(\eta v)}_{L^2(B_{(1+\delta)R}\setminus\Gamma)}^2,
\end{align*}
where, in the last step, we applied Young's inequality
as well as the assumptions $\frac{h}{R} \leq \frac{\delta}{8}$ and $\delta R \leq 2 \diam (\Omega)$  multiple times.
The last term in \eqref{eq:lem:Caccioppoli-10HS} can
be estimated with \eqref{eq:NIapprox}, $\eta \leq 1$,
the previous estimates \eqref{eq:NIjumpest} -- \eqref{eq:traceestimate}, 
and the assumption 
$\frac{h}{R} \leq \frac{\delta}{8}$, as well as $\delta R \leq 2 \diam (\Omega)$ by
\bea
\abs{\mu\skp{I_h(\eta^2 \vjump),1}}&\lesssim& \abs{\mu\skp{\eta^2 \vjump,1}} +
\abs{\mu\skp{\eta^2\vjump - I_h(\eta^2 \vjump),1}}  \\
&\lesssim&\abs{\mu}\abs{B_{(1+\delta)R}\cap\Gamma}^{1/2}
\left(\norm{\eta^2\vjump}_{L^2(\Gamma)}+\norm{\eta^2\vjump - I_h(\eta^2\vjump)}_{L^2(\Gamma)}\right) \\
&\lesssim&
\abs{\mu}((1+\delta)R)^{(d-1)/2}\left(\norm{\eta^2\vjump}_{L^{2}(\Gamma)}+
h^{1/2}\norm{\eta^2\vjump}_{H^{1/2}(\Gamma)}\right) \\
 &\lesssim&
\abs{\mu}((1+\delta)R)^{(d-1)/2}\left(\norm{v}_{L^2(B_{(1+\delta)R}\setminus\Gamma)} + 
\norm{\nabla(\eta v)}_{L^2(B_{(1+\delta)R}\setminus\Gamma)}\right).
\eea
Applying Young's inequality, we obtain
\bee
\abs{\mu\skp{I_h(\eta^2 \vjump),1}}\leq 
C((1+\delta)R)^{d-1}\abs{\mu}^2 + C\frac{1}{(\delta R)^2}\norm{v}_{L^2(B_{(1+\delta)R}\setminus\Gamma)}^2
+\frac{1}{4}\norm{\nabla(\eta v)}_{L^2(B_{(1+\delta)R}\setminus\Gamma)}^2.
\ee
Inserting the previous estimates in \eqref{eq:lem:Caccioppoli-10HS}, 
Lemma~\ref{lem:estimateun},
Young's inequality, and the assumption $\frac{h}{R} \leq \frac{\delta}{8}$ lead to
\bea
\abs{\langle \partial_n v, \eta^2 \vjump\rangle} &\leq& 
\abs{\langle \partial_n v, \eta^2 \vjump - I_h (\eta^2 \vjump)\rangle}+\abs{\mu\skp{I_h(\eta^2 \vjump),1}}\\
&\leq&C\frac{h^2}{(\delta R)^2}\norm{\nabla v}_{L^2(B_{(1+\delta)R}\setminus \Gamma)}^2 + 
C\frac{1}{(\delta R)^2}\norm{v}_{L^2(B_{(1+\delta)R}\setminus \Gamma)}^2 +C((1+\delta)R)^{d-1}\abs{\mu}^2 \\
& & + \frac{1}{2} \norm{\nabla(\eta v)}_{L^2(B_{(1+\delta)R}\setminus\Gamma)}^2.
\eea
Inserting this in \eqref{eq:CaccHS1} and subtracting the term 
$\frac{1}{2} \norm{\nabla(\eta v)}_{L^2(B_{(1+\delta)R}\setminus\Gamma)}^2$ from both sides 
finally leads to
\bea
\norm{\nabla(\eta v)}_{L^2(B_{(1+\delta)R}\setminus \Gamma)}^2 \lesssim \frac{h^2}{(\delta R)^2} \norm{\nabla v}^2_{L^2(B_{(1+\delta)R} \setminus \Gamma)} + 
\frac{1}{(\delta R)^2}\norm{v}_{L^2(B_{(1+\delta)R} \setminus \Gamma)}^2+((1+\delta)R)^{d-1}\abs{\mu}^2,
\eea
which finishes the proof.
\epro

We consider $\gamma$-shape regular triangulations ${\mathcal{E}_H}$ of $\R^d$ that conform to 
$\Omega$. More precisely, we will assume that every $E \in {\mathcal{E}_H}$ satisfies either 
$E \subset \ol{\Omega}$ or $E\subset \Omega^c$ and that the restrictions 
${\mathcal E}_H|_{\Omega}$ and ${\mathcal E}_H|_{\Omega^c}$ are $\gamma$-shape regular, regular
triangulations of $\Omega$ and $\Omega^c$ of mesh size $H$, respectively. 
On the piecewise regular mesh ${\mathcal E}_H$, we define the Scott-Zhang projection 
$J_H:H^1(\R^d\setminus \Gamma) \rightarrow 
S^{1,1}_{pw}:= \{v\,:\, v|_\Omega \in S^{1,1}({\mathcal E}_H|_{\Omega})\ \mbox{ and } \ 
v|_{\Omega^c} \in S^{1,1}({\mathcal E}_H|_{\Omega^c})\}$ in a piecewise fashion by 
\been\label{eq:pwInterpolation}
J_H v = \left\{
\begin{array}{l}
 \widetilde J_H^{\rm int} v \quad \text{for} \, x \in \overline{\Omega}, \\
 \widetilde J_H^{\rm ext} v \quad \textrm{otherwise};
 \end{array}
 \right.
\een
here, $\widetilde J_H^{\rm int}$, $\widetilde J_H^{\rm ext}$ denote the Scott-Zhang projections 
for the grids $\mathcal{E}_H|_{\Omega}$ and $\mathcal{E}_H|_{{\Omega}^c}$. Since $J_H$ is a piecewise Scott-Zhang projection
the approximation properties proved in \cite{ScottZhang} apply and result in the following estimates: 
\begin{equation}
\label{eq:SZapprox}
\norm{v-J_H v}_{H^m(E)}^2 \leq C H^{2(\ell-m)} 
\begin{cases} 
\abs{v}_{H^{\ell}(\omega_E^\Omega)} &\mbox{ if $E \subset \Omega$} \\
\abs{v}_{H^{\ell}(\omega_E^{\Omega^c})} &\mbox{ if $E \subset \Omega^c$} 
\end{cases}
\quad 0 \leq m \leq \ell \leq 1;
\end{equation}
here, 
\begin{align*}
\omega_E^\Omega = 
\bigcup\left\{E' \in \mathcal{E}_H|_{\Omega} \;:\; E \cap E' \neq \emptyset \right\},  
\qquad 
\omega_E^{\Omega^c} = 
\bigcup\left\{E' \in \mathcal{E}_H|_{\Omega^c} \;:\; E \cap E' \neq \emptyset \right\}. 
\end{align*}
The constant $C>0$ in (\ref{eq:SZapprox}) depends only on the $\gamma$-shape regularity of the 
quasiuniform triangulation $\mathcal{E}_H$ and the dimension $d$.\\

Let $\Pi_{h,R,\mu} : (H^1(B_R\setminus\Gamma),\triplenorm{\cdot}_{h,R}) \rightarrow 
(\H_{h,0}(B_R,\Gamma_{\rho},\mu), 
\triplenorm{\cdot}_{h,R})$ 
be the orthogonal projection, 
which is well-defined since $\H_{h,0}(B_R,\Gamma_{\rho},\mu)\subset H^1(B_R\setminus\Gamma)$ is 
a closed subspace by Lemma~\ref{lem:closedsubspaceHS}. 

\begin{lemma}\label{lem:lowdimappHS}
Let $\delta \in (0,1)$,  $R\in (0,2 \operatorname*{diam} (\Omega))$ 
be such that $\frac{h}{R}\leq \frac{\delta}{8}$. Let 
$B_R$, $B_{(1+\delta)R}$, $B_{(1+2\delta)R}$ be concentric boxes. 
Let $\Gamma_{\rho}\subset \Gamma$ be of the form \eqref{eq:screen}
and $\mu \in\R$.
Let $\mathcal{E}_H $ be an (infinite) $\gamma$-shape regular triangulation of $\mathbb{R}^d$ 
of mesh width $H$ that conforms to $\Omega$ as described above. Assume $\frac{H}{R} \leq \frac{\delta}{4}$. 
Let $J_H: H^1(\mathbb{R}^d\setminus\Gamma) \rightarrow S^{p,1}_{\rm pw}$ be the piecewise Scott-Zhang projection
defined in \eqref{eq:pwInterpolation}. 
Then, there exists a constant $C_{\rm app} > 0$ that depends only on $\Omega$, $d,p$, and $\gamma$, such that for
$v\in\H_{h,0}(B_{(1+2\delta)R},\Gamma_{\rho},\mu)$
\begin{enumerate}[(i)]
\item 
\label{item:lem:lowdimapp-ii}
$\big(v-\Pi_{h,R,\mu}J_H v\big)|_{B_{R}} \in \H_{h,0}(B_{R},\Gamma_{\rho},0)$;  \\[-2mm]
\item 
\label{item:lem:lowdimapp-i}
$\triplenorm{v-\Pi_{h,R,\mu}J_H v}_{h,R} \leq C_{\rm app} 
\left(\frac{h}{R}+\frac{H}{R}\right)\left(\frac{1+2\delta}{\delta}\triplenorm{v}_{h,(1+2\delta)R} + 
((1+2\delta)R)^{(d-1)/2}\abs{\mu}\right)$;  
\item 
\label{item:lem:lowdimapp-iii}
$\dim W\leq C_{\rm app}\left(\frac{(1+2\delta)R}{H}\right)^d$, where 
$W:=\Pi_{h,R,\mu}J_H \H_{h,0}(B_{(1+2\delta)R},\Gamma_{\rho},\mu) $. 
\end{enumerate}
\end{lemma}

\bpro
For $u \in \mathcal{H}_{h,0}(B_{(1+2\delta)R},\Gamma_{\rho},\mu)$, we have 
$u \in \mathcal{H}_{h,0}(B_{R},\Gamma_{\rho},\mu)$ as well and hence 
$\Pi_{h,R,\mu}\left(u|_{B_{R}}\right) = u|_{B_{R}}$, which gives (\ref{item:lem:lowdimapp-ii}).

The assumption $\frac{H}{R} \leq \frac{\delta}{4}$ implies 
$\bigcup\{E \in \mathcal{E}_H \;:\; \omega_E \cap B_{R} \neq \emptyset\} \subseteq B_{(1+\delta)R}$.
The locality and the approximation properties \eqref{eq:SZapprox} of $J_H$ yield
\begin{eqnarray*}
\frac{1}{H} \norm{u - J_Hu}_{L^2(B_{R}\setminus\Gamma)} + 
\norm{\nabla(u - J_Hu)}_{L^2(B_{R}\setminus\Gamma)} &\lesssim&  
\norm{\nabla u}_{L^2(B_{(1+\delta)R}\setminus\Gamma)}.
\end{eqnarray*}
We apply Lemma~\ref{lem:CaccioppoliHS} with
$\widetilde{R} = (1+\delta)R$ and $\widetilde{\delta} = \frac{\delta}{1+\delta}$. 
Note that $(1+\widetilde{\delta})\widetilde{R} = (1+2\delta)R$, and
 $\frac{h}{\widetilde{R}}\leq \frac{\widetilde{\delta}}{8}$
follows from $8h \leq \delta R = \widetilde{\delta}\widetilde{R}$. Hence, we obtain
\begin{align*}
&\triplenorm{u-\Pi_{h,R,\mu}J_H u}_{h,R}^2 =\triplenorm{\Pi_{h,R,\mu}\left(u-J_H u\right)}^2_{h,R} \leq 
\triplenorm{u-J_H u}_{h,R}^2 \\
& \qquad= \left(\frac{h}{R}\right)^{2}\norm{\nabla (u-J_H u)}_{L^2(B_{R}\setminus\Gamma)}^2   + 
\frac{1}{R^2} \norm{u-J_H u}_{L^2(B_{R}\setminus\Gamma)}^2\\
&\qquad\lesssim\frac{h^2}{R^2}\norm{\nabla u}_{L^{2}(B_{(1+\delta) R}\setminus\Gamma)}^2 + 
\frac{H^2}{R^2}\norm{\nabla u}_{L^2(B_{(1+\delta)R}\setminus\Gamma)}^2\\
  &\qquad\lesssim \left(\frac{h}{R}+\frac{H}{R}\right)^2
\left(\frac{(1+2\delta)^2}{\delta^2}\triplenorm{u}^2_{h,(1+2\delta)R}+((1+2\delta)R)^{d-1}\abs{\mu}^2\right),
\end{align*}
which concludes the proof (\ref{item:lem:lowdimapp-i}).
The statement (\ref{item:lem:lowdimapp-iii}) follows from the fact that 
$\dim J_H\mathcal{H}_{h,0}(B_{(1+2\delta)R},\Gamma_{\rho},\mu) \lesssim ((1+2\delta)R/H)^d$. 
\epro

\begin{lemma}\label{cor:lowdimappHS}
Let $C_{\rm app}$ be the constant of Lemma~\ref{lem:lowdimappHS}.
Let $q,\kappa \in (0,1)$,  $R \in (0,2\operatorname*{diam}(\Omega))$, $k \in \mathbb{N}$, 
and $\Gamma_{\rho}\subset \Gamma$ be of the form \eqref{eq:screen}. 
Assume 
\begin{equation}
\label{eq:cor:lowdimapp-1HS}
\frac{h}{R} \leq \frac{\kappa q} {32 k \max\{C_{\rm app},1\}}.
\end{equation}
Then, there exists a finite dimensional subspace $\widehat{W}_k$ of 
$\mathcal{H}_{h,0}(B_{(1+\kappa)R},\Gamma_{\rho},\mu)$ 
with dimension 
$$
\dim \widehat{W}_k \leq C_{\rm dim} \left(\frac{1 + \kappa^{-1}}{q}\right)^dk^{d+1},$$
such that for every $v \in \mathcal{H}_{h,0}(B_{(1+\kappa)R},\Gamma_{\rho},\mu)$ it holds 

\begin{align}
\label{eq:lowdimappHS}
&\min_{\widehat{w} \in \widehat{W}_k} \norm{[\gamma_0 v]- [\gamma_0\widehat{w}]}_{L^2(B_R\cap\Gamma_{\rho})}  \\
&
\qquad \leq  C_{\rm low}R(1+\kappa) h^{-1/2} \min_{\widehat{w}\in \widehat{W}_k} 
\triplenorm{v-\widehat{w}}_{h,(1+\kappa/2)R} \nonumber\\
&
\qquad \leq 
C_{\rm low}R(1+\kappa) h^{-1/2} q^{k} \left(\triplenorm{v}_{h,(1+\kappa)R}+
((1+\kappa)R)^{(d-1)/2}\abs{\mu}\right). 
\nonumber 
\end{align}

The constants $C_{\rm dim}$, $C_{\rm low}>0$ depends only on $\Omega$, $d$, $p$, and the 
$\gamma$-shape regularity of the quasiuniform triangulation $\mathcal{T}_h$.
\end{lemma}

\bpro
Let $B_R$ and $B_{(1+\delta_j)R}$ 
with $\delta_j := \kappa(1-\frac{j}{2k})$ for $j=0,\dots,k$ be concentric boxes. 
We note that $\kappa = \delta_0>\delta_1 > \dots >\delta_k = \frac{\kappa}{2}$.
We choose $H = \frac{\kappa q R}{32k\max\{C_{\rm app},1\}}$, where $C_{\rm app}$ is the constant 
in Lemma~\ref{lem:lowdimappHS}.
By the choice of $H$, we have $h \leq H$. We apply Lemma~\ref{lem:lowdimappHS} with 
$\widetilde{R}\! =\! (1+\delta_j)R$ and 
$\widetilde{\delta}_j\! =\! \frac{\kappa}{4k(1+\delta_j)}\!<\!\frac{1}{4}$.
Note that $\delta_{j-1} = \delta_j +\frac{\kappa}{2k}$ gives 
$(1+\delta_{j-1})R=(1+2\widetilde{\delta}_j)\widetilde{R}$. 
Our choice of $H$ implies 
$\frac{H}{\widetilde{R}}\leq \frac{\widetilde{\delta}_j}{4}$. 
Hence, for $j=1$, Lemma~\ref{lem:lowdimappHS} provides a 
subspace $W_1$ of $\mathcal{H}_{h,0}(B_{(1+\delta_1)R},\Gamma_{\rho},\mu)$ 
with $\dim W_1 \leq C\left(\frac{(1+\kappa)R}{H}\right)^d$ 
and a $w_1 \in W_1$ such that
\begin{eqnarray*}
\triplenorm{v-w_1}_{h,(1+\delta_1)R} &\leq& 
2C_{\rm app}\frac{H}{(1+\delta_1)R}\left({\frac{1+2\widetilde{\delta}_1}{\widetilde{\delta}_1}} 
\triplenorm{v}_{h,(1+\delta_0)R}+((1+\delta_0)R)^{(d-1)/2}\abs{\mu}\right)\\
 &=& 8C_{\rm app}\frac{k H}{\kappa R}(1+2\widetilde{\delta}_1)\left(\triplenorm{v}_{h,(1+\kappa)R} + 
\frac{\widetilde{\delta}_1}{1+2\widetilde{\delta}_1}(1+\delta_0)^{(d-1)/2}\abs{\mu}\right) \\
&\leq& q\left(\triplenorm{v}_{h,(1+\kappa)R}+((1+\kappa)R)^{(d-1)/2}\abs{\mu}\right).
\end{eqnarray*}
Since $v-w_1 \in \mathcal{H}_{h,0}(B_{(1+\delta_1)R},\Gamma_{\rho},0)$, we can use Lemma~\ref{lem:lowdimappHS} again
(this time with $\mu = 0$) 
and get an approximation $w_2$ of $v-w_1$ in a subspace $W_2$ 
of $\mathcal{H}_{h,0}(B_{(1+\delta_1)R},\Gamma_{\rho},0)$ with
$\dim W_2\leq C\left(\frac{(1+\kappa)R}{H}\right)^d$. Arguing as for $j=1$, we get
\begin{equation*}
\triplenorm{v-w_1-w_2}_{h,(1+\delta_2)R} \leq q \triplenorm{v-w_1}_{h,(1+\delta_1)R} \leq  
q^2 \left(\triplenorm{v}_{h,(1+\kappa)R}+((1+\kappa)R)^{(d-1)/2}\abs{\mu}\right).
\end{equation*}
Continuing this process $k-2$ times leads to an approximation $\widehat{w} := \sum_{j=1}^kw_i$ 
in the space $\widehat{W}_k := \sum_{j=1}^{k}W_j$ 
of dimension $\dim \widehat{W}_k \leq Ck\left(\frac{(1+\kappa)R}{H}\right)^d=
C_{\rm dim} ((1+\kappa^{-1}) q^{-1})^dk^{d+1}$ such that
\begin{equation} \label{eq:tmp:iterationargument}
\triplenorm{v-\widehat{w}}_{h,(1+\kappa/2)R}=\triplenorm{v-\widehat{w}}_{h,(1+\delta_k)R} 
\leq q^k\left(\triplenorm{v}_{h,(1+\kappa)R}+((1+\kappa)R)^{(d-1)/2}\abs{\mu}\right).
\end{equation}


The last step of the argument is to use the multiplicative trace inequality. 
With a suitable cut-off function $\eta$ supported by $B_{(1+\kappa/2)R}$ and 
$\|\nabla \eta\|_{L^\infty} \lesssim (\kappa R)^{-1}$ as well as $\eta \equiv 1$ on $B_R$, we get 
for $z \in H^1(B_{(1+\kappa/2)R}\setminus\Gamma)$ 
\begin{align*}
\|[\gamma_0 z]\|^2_{L^2(B_R\cap\Gamma)} &\leq  \|[\gamma_0 (\eta z)]\|^2_{L^2(\Gamma)} 
\lesssim \|\eta z\|_{L^2(\R^d\setminus\Gamma)}
\|\eta z\|_{H^1(\R^d\setminus\Gamma)} \\
& \lesssim \frac{1}{\kappa R}\|z\|^2_{L^2(B_{(1+\kappa/2)R})} + 
 \|z\|_{L^2(B_{(1+\kappa/2)R})} \|\nabla z\|_{L^2(B_{(1+\kappa/2)R}\setminus\Gamma)} \\
& \lesssim \frac{1}{\kappa R}\|z\|^2_{L^2(B_{(1+\kappa/2)R})} + 
 h^{-1}\|z\|_{L^2(B_{(1+\kappa/2)R})}^2+ h\|\nabla z\|_{L^2(B_{(1+\kappa/2)R}\setminus\Gamma)}^2 \\
&\lesssim \left((1+\kappa/2)R\right)^2h^{-1} \triplenorm{z}^2_{h,(1+\kappa/2)R},
\end{align*}
where the last step follows from the assumption $\frac{h}{\kappa R}\leq 1$.
Using this estimate for $z=v-\widehat{w}$ together with \eqref{eq:tmp:iterationargument} gives 
\begin{align*}
\min_{\widehat w \in \widehat W_k} 
\|[\gamma_0 v] - [\gamma_0 \widehat w]\|_{L^2(B_R \cap \Gamma)} 
&\leq C_{\rm low} 
(1 + \kappa)R \, h^{-1/2} q^k \left[
\triplenorm{v}_{h,(1+\kappa)R} + \left((1+\kappa)R\right)^{(d-1)/2}|\mu| 
\right].
\end{align*}
This concludes the proof.
\epro

\begin{remark}
{\rm 
The proof of Lemma~\ref{cor:lowdimappHS} shows that approximation results in $H^{1/2}$ can be achieved at
the expense of an additional factor $h^{-1/2}$: With the cut-off function $\eta$ that is used at the end of the 
proof of Lemma~\ref{cor:lowdimappHS}, we can can bound for $z \in H^1(B_{(1+\kappa/2)R}\setminus\Gamma)$
$$
\|[\gamma_0 (\eta z)]\|_{H^{1/2}(\Gamma)} \lesssim 
\|\eta z\|_{H^{1}(\R^d\setminus\Gamma)} \lesssim 
(1 +\kappa/2)R\,  h^{-1} \triplenorm{z}_{h,B_{(1+\kappa/2)R}}
$$
Hence, with the spaces $\widehat W_k$ of Lemma~\ref{cor:lowdimappHS} one gets 
$$
\min_{\widehat w \in \widehat W_k} 
\|[\gamma_0 (\eta (v - \widehat w))]\|_{H^{1/2}(\Gamma)} 
\lesssim C_{\rm low}^\prime (1+\kappa)R\, h^{-1} q^k 
\left[
\triplenorm{v}_{h,(1+\kappa)R} + \left((1+\kappa)R\right)^{(d-1)/2}|\mu| 
\right]. 
$$
}\eremk
\end{remark}

Now we are able to prove the main result of this section.

\bpro[of Theorem~\ref{thm:function-approximationHypSing}]
Choose $\kappa = \frac{1}{1+\eta}$. By assumption, we have 
$\dist(B_{R_\tau},B_{R_\sigma}) 
\ge \eta^{-1} \diam B_{R_\tau} = \sqrt{d} \eta^{-1} R_{\tau}$. 
In particular, this implies 
\bee\dist(B_{(1+\kappa) R_\tau},B_{R_\sigma}) \geq \dist(B_{R_\tau},B_{R_\sigma}) - 
\kappa R_{\tau} \sqrt{d} \geq \sqrt{d}R_{\tau}(\eta^{-1}-\kappa) = 
\sqrt{d}R_{\tau}\left(\frac{1}{\eta} - \frac{1}{1+\eta}\right) >0.\ee

Let $\phi_h \in S^{p,1}(\T_h)$ solve (\ref{eq:modelHS}). Recall from (\ref{eq:discrete-stability}) that 
\begin{equation*}
\norm{\phi_h}_{H^{1/2}(\Gamma)} + |\lambda| \lesssim \norm{\Pi^{L^2} f}_{L^2(\Gamma)}. 
\end{equation*}

The potential $u = \widetilde{K}\phi_h$ then satisfies 
$u \in \H_{h,0}(B_{(1+\kappa) R_\tau},\Gamma,\lambda)$. 
Furthermore, the boundedness of $\widetilde{K}: H^{1/2}(\Gamma) \ra H^1_{\text{loc}}(\R^d)$ 
and $\frac{h}{R_{\tau}}<1$ lead to
\bea
\triplenorm{\widetilde{K} \phi_h}_{h,R_{\tau}(1+\kappa)} &\leq& 
2\left(1+\frac{1}{R_\tau}\right)\norm{\widetilde{K} \phi_h}_{H^1(B_{2R_{\tau}})} \\&\lesssim& 
\left(1+\frac{1}{R_\tau}\right) \norm{\phi_h}_{H^{1/2}(\Gamma)} 
\lesssim \left(1+\frac{1}{R_\tau}\right)\norm{\Pi^{L^2}f}_{L^{2}(\Gamma)}.
\eea
We are now in position to define the space $W_k$, for which we distinguish two cases. 
\newline
{\bf Case 1:} The condition \eqref{eq:cor:lowdimapp-1HS} is satisfied with $R = R_{\tau}$. 
With the space $\widehat{W}_k$ provided by Lemma~\ref{cor:lowdimappHS} we set 
$W_k := \{[\gamma_0 \widehat{w}] : \widehat{w} \in \widehat{W}_k\}$. 
Then, Lemma~\ref{cor:lowdimappHS} and $R_{\tau}\leq 2\diam(\Omega)$ as well as 
$\kappa \leq 1$ lead to

 \begin{eqnarray*}
\min_{w\in W_k}\norm{\phi_h-w}_{L^2(B_{R_{\tau}}\cap \Gamma)} &\lesssim&  
(1+\kappa)R_{\tau}\, h^{-1/2} q^k 
\left(\triplenorm{\widetilde{K}\phi_h}_{h,(1+\kappa)R_{\tau}}+\abs{\lambda} \right) \\
 &\lesssim& 
(1+\kappa)(R_{\tau}+1)h^{-1/2}q^k\norm{\Pi^{L^2}f}_{L^{2}(\Gamma)} 
 \lesssim h^{-1/2} q^k \norm{\Pi^{L^2}f}_{L^{2}(\Gamma)} ,
\end{eqnarray*}
and the dimension of $W_k$ is bounded by 
$$
\dim W_k \leq C_{\rm dim} \left(\frac{1+\kappa^{-1}}{q}\right)^d k^{d+1} = C_{\rm dim} (2 + \eta)^d q^{-d} k^{d+1}. 
$$
{\bf Case 2:} The condition \eqref{eq:cor:lowdimapp-1HS} is not satisfied with $R = R_{\tau}$. 
Then, we select 
$W_k:= \left\{w|_{B_{R_\tau}\cap \Gamma} : w \in S^{p,1}({\mathcal T}_h)\right\}$ and 
the minimum in \eqref{eq:thm:function-approximation-1HS} is obviously zero. 
By the choice of $\kappa$ and $\frac{h}{R_{\tau}}>\frac{\kappa q}{32k\max\{C_{\rm app},1\}}$, 
the dimension of $W_k$ is bounded by 
$$
\dim W_k \lesssim \left(\frac{R_{\tau}}{h}\right)^{d-1} 
\lesssim \left(\frac{32 k \max\{C_{\rm app},1\}}{\kappa q}\right)^{d-1} 
\simeq \left((1+\eta) q^{-1} k \right)^{d-1} 
\lesssim (2+\eta)^d q^{-d} k^{d+1}. 
$$
This concludes the proof.
 of the first inequality in \eqref{eq:thm:function-approximation-1HS}.  
The second inequality in \eqref{eq:thm:function-approximation-1HS} follows 
from the $L^2(\Gamma)$-stability of the $L^2(\Gamma)$-orthogonal projection.
\epro

\section{$\H$-matrix approximation}
\label{sec:H-matrix-approximation}
In order to obtain an $\H$-matrix approximating $\Win$ (cf. (\ref{eq:Wtilde}))
we start with the construction of a low-rank approximation of an admissible matrix block.

\begin{theorem}\label{th:blockapprox}
Fix an admissibility parameter $\eta > 0$ and $q\in (0,1)$. 
Let the cluster pair $(\tau,\sigma)$ be $\eta$-admissible. 
Then, for every $k \in \mathbb{N}$, there are matrices
$\mathbf{X}_{\tau\sigma} \in \mathbb{R}^{\abs{\tau}\times r}$, $\mathbf{Y}_{\tau\sigma} \in \mathbb{R}^{\abs{\sigma}\times r}$ of rank $r \leq C_{\rm dim} (2+\eta)^d q^{-d}k^{d+1}$
such that
\begin{equation}
\norm{\Win|_{\tau \times \sigma} - \mathbf{X}_{\tau\sigma}\mathbf{Y}_{\tau\sigma}^T}_2 
\leq C_{\rm apx}   N^{(2d-1)/(2d-2)}  q^k.
\end{equation}
The constants $C_{\rm apx}$, $C_{\rm dim}>0$ depend only on $\Omega$, $d$, 
the $\gamma$-shape regularity of the quasiuniform triangulation $\mathcal{T}_h$, and $p$.
\end{theorem}

\bpro
If $C_{\rm dim} (2+\eta)^d q^{-d} k^{d+1} \geq \min (\abs{\tau},\abs{\sigma})$, we use the exact matrix block 
$\mathbf{X}_{\tau\sigma}=\Win|_{\tau \times \sigma}$ and 
$\mathbf{Y}_{\tau\sigma} = I \in \mathbb{R}^{\abs{\sigma}\times\abs{\sigma}}$. 

If $C_{\rm dim}(2+\eta)^d q^{-d} k^{d+1}  < \min (\abs{\tau},\abs{\sigma})$, we employ the approximation
result of Theorem~\ref{thm:function-approximationHypSing} in the following way.
Let  $\lambda_i:L^2(\Gamma) \rightarrow \mathbb{R}$ be continuous linear functionals on $L^2(\Gamma)$
 satisfying $\lambda_i(\psi_j) = \delta_{ij}$, as well as the stability estimate 
$\norm{\lambda_i(w)\psi_i}_{L^2(\Gamma)} \lesssim \norm{w}_{L^2(\supp \psi_i)}$ for $w \in L^2(\Gamma)$,
where the suppressed constant depends only on the shape-regularity of the quasiuniform mesh $\mathcal{T}_h$.
For the existence of such functionals, we refer to \cite{ScottZhang}.
 We define $\R^{\tau}:= \{\mathbf{x} \in \R^N \; : \; x_i = 0 \; \forall \; i \notin \tau \}$ and the mappings 
 \begin{equation*}
\Lambda_{\tau} : L^2(\Gamma) \rightarrow \mathbb{R}^{\tau}, v \mapsto (\lambda_i(v))_{i \in \tau} \;\text{and} 
\; \Phi_{\tau}: \mathbb{R}^{\tau} \rightarrow S^{p,1}({\mathcal T}_h),\; \mathbf{x} \mapsto \sum_{j\in\tau} x_j \psi_j.
\end{equation*}
The interpolation operator
$\Phi_{\tau}\Lambda_{\tau}$ is, due to our assumptions on the functionals $\lambda_i$, 
stable in $L^2$ and for a piecewise polynomial function $\widetilde{\phi}\in S^{p,1}(\T_h)$ we get 
$\Phi_{\tau}(\Lambda_{\tau} \widetilde{\phi}) = \widetilde{\phi}|_{\Gamma_{\tau}}$
with $\Gamma_{\tau}:=\bigcup_{i \in \tau} \supp \psi_i \subset B_{R_{\tau}}$. 
For $\mathbf{x} \in \R^{\tau}$, \eqref{eq:basisisomorphism} implies
\bee
Ch^{(d-1)/2}\norm{\mathbf{x}}_{2} \leq \norm{\Phi_\tau(\mathbf{x})}_{L^2(\Gamma)} \leq 
\widetilde{C} h^{(d-1)/2}\norm{\mathbf{x}}_{2}, \quad \forall\mathbf{x} \in \R^{{\tau}}. 
\ee
The adjoint $\Lambda_{\mathcal{I}}^* : \R^N \ra L^2(\Gamma)'\simeq L^2(\Gamma), 
\mathbf{b}\mapsto\sum_{i\in \mathcal{I}}b_i\lambda_i$ 
of $\Lambda_{\mathcal{I}}$ satisfies, 
because of \eqref{eq:basisisomorphism} and the $L^2$-stability of $\Phi_{\mathcal{I}} \Lambda_{\mathcal{I}}$,
\bee
\norm{\Lambda_{\mathcal{I}}^* \mathbf{b}}_{L^2(\Gamma)} = \sup_{w \in L^2(\Gamma)}
\frac{\skp{\mathbf{b},\Lambda_{\mathcal{I}} w}_2}{\norm{w}_{L^2(\Gamma)}} 
  \lesssim \norm{\mathbf{b}}_2 \sup_{w \in L^2(\Gamma)}
\frac{h^{-(d-1)/2}\norm{\Phi_{\mathcal{I}}\Lambda_{\mathcal{I}} w}_{L^2(\Gamma)}}{\norm{w}_{L^2(\Gamma)}} 
\leq h^{-(d-1)/2}\norm{\mathbf{b}}_2. 
\ee
Let $\mathbf{b} \in \R^N$. Defining $f := \Lambda_{\mathcal{I}}^*\mathbf{b}|_{\sigma}$,
we get $b_i =\skp{f,\psi_i}$ for $i \in \sigma$ and $\supp f \subset B_{R_{\sigma}}\cap\Gamma$.
Theorem~\ref{thm:function-approximationHypSing} provides a 
finite dimensional space $W_k$ and an element $w \in W_k$ 
that is a good approximation to the Galerkin solution $\phi_h|_{B_{R_{\tau}}\cap \Gamma}$.
It is important to note that the space $W_k$ is constructed independently of the function $f$; 
it depends only on the cluster pair $(\tau,\sigma)$. 
The estimate \eqref{eq:basisisomorphism}, the approximation result from 
Theorem~\ref{thm:function-approximationHypSing},
and $\norm{\Pi^{L^2}f}_{L^2(\Gamma)} \leq \norm{f}_{L^2(\Gamma)}
\leq \norm{\Lambda_{\mathcal{I}}^*\mathbf{b}}_{L^2(\Gamma)}
\lesssim h^{-(d-1)/2}\norm{\mathbf{b}}_2$ imply
\bea
\norm{\Lambda_{\tau} \phi_h - \Lambda_{\tau} w}_2 &\lesssim& 
h^{-(d-1)/2}\norm{\Phi_{\tau}(\Lambda_{\tau} \phi_h-\Lambda_{\tau} w)}_{L^2(\Gamma)} 
\leq h^{-(d-1)/2}\norm{\phi_h-w}_{L^2(B_{R_{\tau}}\cap \Gamma)} \\
  &\lesssim&  h^{-(d-1)/2-1/2} \, 
             q^k\norm{\Pi^{L^2}f}_{L^2(\Gamma)}  \lesssim 
 h^{-(2d-1)/2} \,q^k\norm{\mathbf{b}}_{2}.
\eea
In order to translate this approximation result to the matrix level, let 
\bee
\mathcal{W} := \{\Lambda_{\tau} w \; : \; w \in W_k \}.
\ee
Let the columns of $\mathbf{X}_{\tau\sigma}$ be an orthogonal basis of the space $\mathcal{W}$. 
Then, the rank of $\mathbf{X}_{\tau\sigma}$ is bounded by $\dim W_k \leq C_{\rm dim}  (2+\eta)^d q^{-d}k^{d+1} $. 
Since $\mathbf{X}_{\tau\sigma} \mathbf{X}_{\tau\sigma}^T$ is the orthogonal projection from 
$\R^N$ onto $\mathcal{W}$, 
we get that $z:=\mathbf{X}_{\tau\sigma} \mathbf{X}_{\tau\sigma}^T \Lambda_{\tau} \phi_h$ 
is the best approximation of $\Lambda_{\tau} \phi_h$ in $\mathcal{W}$ and arrive at 
\been 
\label{eq:matrix-level-estimate-function}
\norm{\Lambda_{\tau} \phi_h-z}_2 \leq \norm{\Lambda_{\tau} \phi_h-\Lambda_{\tau} w}_2 \lesssim 
 h^{-(2d-1)/2} \, q^k\norm{\mathbf{b}}_2
\simeq  \, N^{(2d-1)/(2d-2)}  q^k\norm{\mathbf{b}}_2.
\een
Note that $\Lambda_{\tau} \phi_h = \Win|_{\tau\times \sigma} \mathbf{b}|_\sigma$.
If we define $\mathbf{Y}_{\tau,\sigma} := \Win|_{\tau\times \sigma}^{T}\mathbf{X}_{\tau\sigma}$, 
we thus get $z = \mathbf{X}_{\tau\sigma} \mathbf{Y}_{\tau\sigma}^T \mathbf{b}|_\sigma$. The bound 
\eqref{eq:matrix-level-estimate-function} expresses 
\begin{equation}
\label{eq:matrix-level-estimate-function-1}
\norm{\left(\Win|_{\tau\times \sigma} - \mathbf{X}_{\tau\sigma} \mathbf{Y}_{\tau\sigma}^T \right)\mathbf{b}|_\sigma}_2 
= \norm{\Lambda_{\tau} \phi_h-z}_2
\lesssim  \, N^{(2d-1)/(2d-2)}  q^k\norm{\mathbf{b}}_2.
\end{equation}
The space $W_k$ depends only on the cluster pair $(\tau,\sigma)$, and the estimate 
\eqref{eq:matrix-level-estimate-function-1} is valid for any $\mathbf b$. 
This concludes the proof. 
\epro

The following lemma gives an estimate for the global spectral norm by the local spectral norms, 
which we will use in combination with Theorem \ref{th:blockapprox} to derive our main result, 
Theorem \ref{th:Happrox}. \\

\blem[{\cite{GrasedyckDissertation},\,\cite[Lemma 6.5.8]{HackbuschBuch},\,\cite{BoermBuch}}]\label{lem:spectralnorm}
Let $\mathbf{M} \in \R^{N\times N}$ and $P$ be a partitioning of ${\mathcal{I}}\times {\mathcal{I}}$. Then,
\bee
\norm{\mathbf{M}}_2 \leq C_{\rm sp} \left(\sum_{\ell=0}^{\infty}\max\{\norm{\mathbf{M}|_{\tau\times \sigma}}_2 : (\tau,\sigma) \in P, \level(\tau) = \ell\}\right),
\ee
where the sparsity constant $C_{\rm sp}$ is defined in \eqref{eq:sparsityConstant}.
\elem

Now we are able to prove our main result, Theorem \ref{th:Happrox}.


\bpro[of Theorem \ref{th:Happrox}]
Theorem \ref{th:blockapprox} provides matrices $\mathbf{X}_{\tau\sigma} \in \R^{\abs{\tau}\times r}$, $\mathbf{Y}_{\tau\sigma} \in \R^{\abs{\sigma}\times r}$, 
so we can define the $\H$-matrix $\mathbf{W}_{\H}$ by 
\bee
\mathbf{W}_{\H} = \left\{
\begin{array}{l}
 \mathbf{X}_{\tau\sigma}\mathbf{Y}_{\tau\sigma}^T \quad \;\textrm{if}\hspace{2mm} (\tau,\sigma) \in P_{\text{far}}, \\
 \Win|_{\tau \times \sigma} \quad \textrm{otherwise}.
 \end{array}
 \right.
\ee
On each admissible block $(\tau,\sigma) \in \Pfar$ we can use the blockwise estimate of Theorem \ref{th:blockapprox} 
and get
\bee
\norm{(\Win - \mathbf{W}_{\H})|_{\tau \times \sigma}}_2 \leq 
C_{\rm apx}  N^{(2d-1)/(2d-2)} q^k.
\ee
On inadmissible blocks, the error is zero by definition. 
Therefore, Lemma \ref{lem:spectralnorm} leads to
\bea
\norm{\Win - \mathbf{W}_{\H}}_2 &\leq& C_{\rm sp} 
\left(\sum_{\ell=0}^{\infty}\text{max}\{\norm{(\Win - \mathbf{W}_{\H})|_{\tau \times \sigma}}_2 : (\tau,\sigma) \in P, \level(\tau) = \ell\}\right) \\
 &\leq& C_{\rm apx}C_{\rm sp}  N^{(2d-1)/(2d-2)} q^k {\rm depth}(\mathbb{T}_{\mathcal{I}}).
\eea
With $r = C_{\rm dim}(2+\eta)^dq^{-d}k^{d+1}$,
the definition $b = -\frac{\ln(q)}{C_{\rm dim}^{1/(d+1)}}q^{d/(d+1)}(2+\eta)^{-d/(1+d)} > 0$ 
leads to $q^k = e^{-br^{1/(d+1)}}$, and hence
\begin{equation*}
\norm{\Win - \mathbf{W}_{\mathcal{H}}}_2 \leq C_{\rm apx}C_{\rm sp} 
 N^{(2d-1)/(2d-2)}  {\rm depth}(\mathbb{T}_{\mathcal{I}})e^{-br^{1/(d+1)}},
\end{equation*}
which concludes the proof.
\epro

\section{Stabilized Galerkin discretization}
\label{sec:stabGalerkin}
In the previous section, we studied a saddle point formulation of the hyper-singular integral operator. 
It is possible to reformulate the hyper-singular integral equation as a positive definite system by 
a rank-one correction that does not alter the solution. In numerical computations, this reformulation is
often preferred, and we therefore study it. Furthermore, it will be the starting point for the $\H$-matrix
Cholesky factorization studied in Section~\ref{sec:LU-decomposition} below. 

The {\em stabilized Galerkin matrix} $\mathbf W^{\text{st}}\in\mathbb R^{N\times N}$ 
is obtained from the matrix $\mathbf W \in \mathbb R^{N \times N}$ as follows: 
\been\label{eq:MatrixHypsing} 
\mathbf W^{\text{st}}_{jk} = \langle W\psi_k,\psi_j\rangle + \alpha\skp{\psi_k,1}\skp{\psi_j,1} = 
\mathbf{W}_{jk}+ \alpha \mathbf{B}_k\mathbf{B}_j, 
\quad \forall j,k = 1,\dots,N. \een
Here, $\alpha > 0$ is a fixed stabilization parameter. 
The matrix $\mathbf W^{\text{st}}$ is symmetric and positive definite.
With the notation from \eqref{eq:matrixHypsing} the stabilized matrix $\mathbf W^{\text{st}}$ can be written as
\bee
\mathbf W^{\text{st}} = \mathbf{W} + \alpha \mathbf{B}\mathbf{B}^T.
\ee
The interest in the stabilized matrix $\mathbf W^{\text{st}}$ arises from the fact that 
solving the linear system 
\bee
\boldsymbol{\mathcal{W}}\bpm \mathbf{x} \\ \lambda \epm := 
\bpm \mathbf{W}  & \mathbf{B} \\ \mathbf{B}^T & 0 \epm \bpm\mathbf{x} \\ \lambda \epm 
= \bpm \mathbf{b} \\0\epm
\ee
is equivalent to solving the symmetric positive definite system 
\been\label{eq:stabGalerkinSystem}
\widehat{\boldsymbol{\mathcal W}}\bpm \mathbf{x} \\ \lambda \epm:=
\bpm \mathbf{W} +\alpha \mathbf{B}\mathbf{B}^T & \mathbf{B} \\ \mathbf{B}^T & 0 \epm \bpm\mathbf{x} \\ \lambda \epm 
= \bpm \mathbf{b} \\0\epm.
\een
For more details about this stabilization, we refer to \cite[Ch. 6.6/12.2]{Steinbach}.

In order to see that the question of approximating $(\mathbf W^{\text{st}})^{-1}$ 
in the $\H$-matrix format is closely related to approximating ${\boldsymbol{\mathcal W}}^{-1}$
in the $\H$-matrix format, we partition 
$$
\Win = \bpm \mathbf{G} & \mathbf{P} \\ \mathbf{P}^T & z \epm
$$
and observe that the inverse $\left(\mathbf W^{\text{st}}\right)^{-1}$ can be computed explicitly: 
\bee
\left(\mathbf W^{\text{st}}\right)^{-1} = 
\mathbf{G} + \left(\mathbf W^{\text{st}}\right)^{-1}\mathbf{B}\mathbf{P}^T.
\ee
Hence, the inverse $\left(\mathbf W^{\text{st}}\right)^{-1}$ can be 
computed just from a rank one update from $\mathbf{G}$, i.e., a subblock of $\Win$. We immediately
get the following corollary to Theorem~\ref{th:Happrox}:  

\begin{corollary}
\label{cor:stabGalerkin}
There exists a blockwise rank-$(r+1)$ approximation 
$\mathbf W^{\text{st}}_\H$ to $({\mathbf W}^{\text{st}})^{-1}$ with 
$$
\|(\mathbf W^{\text{st}})^{-1} - \mathbf W^{\text{st}}_\H\|_2 \leq 
C_{\rm apx} C_{\rm sp} \operatorname*{depth}(\mathbb{T}_{\mathcal I}) 
N^{(2d-1)/(2d-2)} e^{-br^{1/(d+1)}}.
$$
\end{corollary}
%

\section{$\H$-Cholesky decomposition}
\label{sec:LU-decomposition}
In this section we are concerned with proving the existence of a hierarchical Cholesky-decomposition
of the form $\mathbf{W}^{\rm st}\approx\mathbf{C}_{\H}\mathbf{C}_{\H}^T$, where $\mathbf{C}_{\H}$ 
is a lower triangular $\H$-matrix. 
The main results are summarized in Theorem~\ref{th:HLU}. It is shown by 
approximating off-diagonal block of certain Schur complements by low-rank matrices.
Therefore, the main contribution is done in Section~\ref{sec:schur-complements}, 
the remaining steps follow the lines of \cite{Bebendorf07,GrasedyckKriemannLeBorne,FMPFEM}.  \\

The advantage of studying the second system \eqref{eq:stabGalerkinSystem} is that the submatrix 
$\mathbf{W}^{\rm st}=\mathbf{W} +\alpha \mathbf{B}\mathbf{B}^T$ is symmetric and positive definite and therefore has
a Cholesky-decomposition, which can be used to derive a $LU$-decomposition for the whole matrix.
Moreover, the existence of the Cholesky decomposition does not depend on the numbering of the degrees of freedom, 
i.e., for every other numbering of the basis functions there is a Cholesky decomposition as well 
(see, e.g., \cite[Cor.~{3.5.6}]{horn-johnson13}). 
The existence of the Cholesky decomposition implies
the invertibility of the matrix $\mathbf{W}^{\rm st}|_{\rho \times \rho}$ for any $n \leq N$ and index set 
$\rho := \{1,\ldots,n\}$ (see, e.g., \cite[Cor.~{3.5.6}]{horn-johnson13}). 
For the $\H$-Cholesky decomposition of Theorem~\ref{th:HLU} below we assume that 
the unknowns are organized in a binary cluster tree ${\mathbb T}_{\mathcal I}$. This induces 
an ordering of the unknowns by requiring that the unknowns of one of the sons be numbered first and
those of the other son later; the precise numbering for the leaves is immaterial for our purposes. 
This induced ordering of the unknowns allows us to speak of {\em block lower triangular} matrices, if the  
block partition $P$ is based on the cluster tree ${\mathbb T}_{\mathcal I}$. 

The following theorem states that the Cholesky factor $\mathbf{C}$ for the stabilized matrix
can be approximated by a block lower triangular $\H$-matrix and, as a consequence, there exists a
hierarchical $LU$-factorization of $\widehat{\boldsymbol{\mathcal W}}$.
\begin{theorem}\label{th:HLU}
Let $\mathbf{W}^{\rm st} = \mathbf{C}\mathbf{C}^T$ be the Cholesky decomposition.
Let a partition $P$ of $\mathcal{I}\times \mathcal{I}$ be based on a cluster tree
$\mathbb{T}_{\mathcal{I}}$.
Then for every $r\geq 3$, there exist block lower triangular, blockwise rank-$r$ matrices $\mathbf{C_{\H}},\mathbf{L_{\H}}$ 
and a block upper triangular, blockwise rank-$r$ matrix $\mathbf{U_{\H}}$
such that
\begin{enumerate}[(i)]
\item 
\label{item:th:HLU-i}
$\displaystyle \frac{\norm{\mathbf{C}-\mathbf{C_{\H}}}_2}{\norm{\mathbf{C}}_2} \leq 
C_{\rm chol}   N^{\frac{2}{d-1}}  {\rm depth}(\mathbb{T}_{\mathcal{I}})
e^{-br^{1/(d+1)}}$
\item
\label{item:th:HLU-ii}
$\displaystyle\frac{\norm{\mathbf{W}^{\rm st}-\mathbf{C_{\H}}\mathbf{C_{\H}}}_2}
{\norm{\mathbf{W}^{\rm st}}_2}  \leq 
2C_{\rm chol}  N^{\frac{2}{d-1}}  {\rm depth}(\mathbb{T}_{\mathcal{I}}) e^{-br^{1/(d+1)}} 
\!+\! C_{\rm chol}^2  N^{\frac{4}{d-1}} {\rm depth}(\mathbb{T}_{\mathcal{I}})^2 e^{-2br^{1/(d+1)}}$,
\item
\label{item:th:HLU-iii}
$\displaystyle\frac{\norm{\widehat{\boldsymbol{\mathcal W}}-\mathbf{L_{\H}}\mathbf{U_{\H}}}_2}
{\norm{\widehat{\boldsymbol{\mathcal W}}}_2}  \leq 
2C_{\rm chol}  N^{\frac{2}{d-1}}  {\rm depth}(\mathbb{T}_{\mathcal{I}}) e^{-br^{1/(d+1)}} 
\!+\! C_{\rm chol}^2  N^{\frac{4}{d-1}}  {\rm depth}(\mathbb{T}_{\mathcal{I}})^2 e^{-2br^{1/(d+1)}}$,
\end{enumerate}
where $C_{\rm chol} = C_{\rm sp}C_{\rm sc}\sqrt{\kappa_2(\mathbf{W}^{\rm st})}$, 
with the sparsity constant $C_{\rm sp}$ of (\ref{eq:sparsityConstant}), 
the spectral condition number $\kappa_2(\mathbf{W}^{\rm st}) := \norm{\mathbf{W}^{\rm st}}_2 
\norm{{\mathbf{W}^{\rm st}}^{-1}}_2$, 
and a constant $C_{\rm sc}$
depending only on $\Omega$, $d$, $p$, the $\gamma$-shape regularity of the quasiuniform triangulation $\T_h$, 
the admissibility parameter $\eta$  and the stabilization parameter $\alpha$. 
\end{theorem}

\subsection{Schur complements}
\label{sec:schur-complements}
For a cluster pair $(\tau,\sigma)$ and $\rho := \{i\in \mathcal{I} : i < \min(\tau\cup\sigma)\}$,
we define the Schur complement
\been\label{eq:defSchur}
\mathbf{S}(\tau,\sigma) = \mathbf{W}^{\rm st}|_{\tau\times\sigma} - \mathbf{W}^{\rm st}|_{\tau\times \rho}
 (\mathbf{W}^{\rm st}|_{\rho\times \rho})^{-1}\mathbf{W}^{\rm st}|_{\rho\times\sigma}.
\een
As mentioned in \cite{FMPBEM} such a Schur complement can be approximated by using $\H$-arithmetic,
but leads to worse estimates with respect to the rank needed for the approximation than the 
procedure here.
Therefore, we revisit our approach from \cite{FMPBEM} that is based on interpreting Schur complements 
as BEM matrices obtained from certain constrained spaces.

The main result in this section is Theorem~\ref{lem:Schur} below. 
For its proof, we need a degenerate approximation 
of the kernel function $\kappa(x,y) = G(x,y)$ of the single layer operator $V$ given by 
$V\phi(x) := \int_{\Gamma}G(x,y)\phi(y)ds_y$.
This classical result, stated here as a degenerate approximation by Chebyshev interpolation, 
is formulated in the following lemma. A proof can be found in \cite{FMPBEM}. 

\blem\label{lem:lowrankGreensfunction}
Let $\widetilde \eta>0$ and fix $\eta^\prime \in (0,2 \widetilde \eta)$. Then, for every hyper cube 
$B_Y \subset \R^d$, $d \in \{2,3\}$ and closed $D_X \subset \R^d$ with 
$\dist(B_Y,D_X)\geq \widetilde\eta \diam(B_Y)$ the following is true: For every 
$r \in \N$ there exist functions $g_{1,i}$, $g_{2,i}$, $i=1,\ldots,r$ such that 
\been
\norm{\kappa(x,\cdot)-\sum_{i=1}^r g_{1,i}(x) g_{2,i}(\cdot)}_{L^{\infty}(B_{Y})}
\leq C \frac{(1+1/\widetilde \eta)}{\dist(\{x\},B_Y)^{d-2}} (1 + \eta^\prime)^{-r^{1/d}}
\qquad \forall x \in D_X, 
\een
for a constant $C$ that depends solely on the choice of $\eta^\prime  \in (0,2\widetilde\eta)$. 
\elem

The following lemma gives a representation for the Schur complement by
interpreting it as a BEM matrix from a certain constrained space.
A main message of the following lemma is that by slightly modifying the Schur complement
$\mathbf{S}(\tau,\sigma)$, we can use an orthogonality without the stabilization term.

\begin{lemma}[Schur complement and orthogonality]
\label{rem:SchurRepresentation}
Let $(\tau,\sigma)$ be an admissible cluster pair,
$\rho := \{i\in \mathcal{I} : i < \min(\tau\cup\sigma)\}$, 
and the Schur complement $\mathbf{S}(\tau,\sigma)$ defined by \eqref{eq:defSchur}.
Let the function
 $\widetilde{\phi} \in S^{p,1}(\T_h)$
with $\widetilde{\phi} = \phi + \phi_{\rho}$, where
$\phi \in S^{p,1}(\T_h), \supp \phi \subset \overline{\Gamma_{\tau}}$ and
$\phi_{\rho} \in S^{p,1}(\T_h), \supp \phi_{\rho} \subset \overline{\Gamma_{\rho}}$, with 
$\Gamma_{\tau},\Gamma_{\rho}$ of the form \eqref{eq:screen},
 satisfy the orthogonality
\been\label{eq:SchurOrthoNoStab}
\skp{W\widetilde{\phi},\widehat{\psi}}_{L^2(\Gamma)} = 0 \quad 
\forall \widehat{\psi} \in S^{p,1}(\T_h) \; \text{with} \; \supp \widehat{\psi} \subset \overline{\Gamma_{\rho}}.
\een
Then, there exists a matrix $\mathbf{D}$ of rank $2$, which is independent of $\phi$ and $\psi$, such that  
\bee
\skp{W\widetilde{\phi},\psi}+\alpha\skp{\widetilde{\phi},1}\skp{\psi,1} = 
\boldsymbol{\phi}^T\left(\mathbf{S}(\tau,\sigma)+\mathbf{D}\right)\boldsymbol{\psi}.
\ee
\end{lemma}
\begin{proof}
Given $\phi$, $\widetilde{\phi}$ is indeed uniquely defined:
By definition of $\widetilde{\phi}$, we get with the matrix $\mathbf{W}$ from \eqref{eq:matrixHypsing}
\bee
0 = \skp{W\widetilde{\phi},\widehat{\psi}}_{L^2(\Gamma)} = 
\skp{W(\phi + \phi_{\rho}),\widehat{\psi}}_{L^2(\Gamma)} = 
(\boldsymbol{\phi}^T\mathbf{W}|_{\tau \times \rho} + \boldsymbol{\phi}_{\rho}^T\mathbf{W}|_{\rho \times \rho})
\widehat{\boldsymbol \psi},
\ee
for $\widehat{\psi} \in S^{p,1}(\T_h)$, $\supp \widehat{\psi} \subset \overline{\Gamma_{\rho}}$ 
and corresponding vector $\widehat{\boldsymbol\psi} \in \R^{\abs{\rho}}$.
Due to $\rho \subsetneq I$, the matrix $\mathbf{W}|_{\rho \times \rho}$ is symmetric and positive definite 
and therefore invertible.
This leads to 
\bee
\boldsymbol{\phi}_{\rho}^T = - \boldsymbol{\phi}^T\mathbf{W}|_{\tau \times \rho}\mathbf{W}|_{\rho \times \rho}^{-1}. 
\ee
Thus, we get for $\psi$ with $\supp \psi \subset \overline{\Gamma_{\sigma}}$
and the vector $\mathbf{B}$ from \eqref{eq:matrixHypsing} that
\bean\label{eq:Schur1}
\skp{W\widetilde{\phi},\psi}+\alpha\skp{\widetilde{\phi},1}\skp{\psi,1} &=& 
\boldsymbol{\phi}^T \left(\mathbf{W}|_{\tau \times \sigma} +
 \alpha\mathbf{B}\mathbf{B}^T|_{\tau \times \sigma} \right)\boldsymbol{\psi}+ 
\boldsymbol{\phi}_{\rho}^T \left(\mathbf{W}|_{\rho \times \sigma}+ 
\alpha\mathbf{B}\mathbf{B}^T|_{\rho \times \sigma}\right) \boldsymbol{\psi} \nonumber\\
&=& \boldsymbol{\phi}^T\left(\mathbf{W}^{\rm st}|_{\tau\times\sigma} - 
\mathbf{W}|_{\tau \times \rho}\mathbf{W}|_{\rho \times \rho}^{-1}\mathbf{W}^{\rm st}|_{\rho\times\sigma} \right)
\boldsymbol{\psi}.
\eean
With the Sherman-Morrison-Woodbury formula (\cite[Ch.~{0.7.4}]{horn-johnson13}), 
the Schur complement $\mathbf{S}(\tau,\sigma)$ can be written as 
\bean\label{eq:Schur2}
\mathbf{S}(\tau,\sigma) &=& \mathbf{W}^{\rm st}|_{\tau\times\sigma} - \mathbf{W}^{\rm st}|_{\tau\times \rho}
 (\mathbf{W}^{\rm st}|_{\rho\times \rho})^{-1}\mathbf{W}^{\rm st}|_{\rho\times\sigma} \nonumber \\
&=& \mathbf{W}^{\rm st}|_{\tau\times\sigma} - \left(\mathbf{W}|_{\tau \times \rho} +
 \alpha\mathbf{B}\mathbf{B}^T|_{\tau \times \rho} \right)\left(\mathbf{W}|_{\rho \times \rho}^{-1} + \mathbf{P}\right)
\mathbf{W}^{\rm st}|_{\rho\times\sigma},
\eean
where $\mathbf{P}$ is a rank one matrix given by 
$\mathbf{P} = \mathbf{W}|_{\rho\times\rho}^{-1}\alpha\mathbf{B}|_{\rho}
\left(1+\alpha\mathbf{B}|_{\rho}^T\mathbf{W}|_{\rho\times\rho}^{-1}\mathbf{B}|_{\rho}\right)
\mathbf{B}|_{\rho}^T\mathbf{W}|_{\rho\times\rho}^{-1}$. 
Thus, comparing the matrices in \eqref{eq:Schur1} and \eqref{eq:Schur2}, we observe that
\bee
\skp{W\widetilde{\phi},\psi}+\alpha\skp{\widetilde{\phi},1}\skp{\psi,1} = 
\boldsymbol{\phi}^T\left(\mathbf{S}(\tau,\sigma)+\mathbf{D}\right)\boldsymbol{\psi},
\ee
with a rank-2 matrix $\mathbf{D}$.
\end{proof}

Now, we are able to prove the main result of this subsection, an approximation result for the 
Schur-complement $\mathbf{S}(\tau,\sigma)$.

\begin{theorem}\label{lem:Schur}
Let $(\tau,\sigma)$ be an $\eta$-admissible cluster pair, 
set $\rho := \{i\in \mathcal{I} : i < \min(\tau\cup\sigma)\}$,

and let the Schur complement $\mathbf{S}(\tau,\sigma)$ be defined in \eqref{eq:defSchur}.
Then for every $r \ge 3$, there exists a rank-$r$ matrix $\mathbf{S}_{r}(\tau,\sigma)$ such that 
\bee
\norm{\mathbf{S}(\tau,\sigma) - \mathbf{S}_{r}(\tau,\sigma)}_2 \leq C_{\rm sc}^\prime h^{d-3}  
e^{-br^{1/(d+1)}},
\ee
where the constants $C_{\rm sc}^\prime$, $b >0$ depend only on $\Omega$,
$d,p$, the $\gamma$-shape regularity of the quasiuniform triangulation $\mathcal{T}_h$, and $\eta$.
Furthermore, there exists a constant $C_{\rm sc}$ depending additionally on the stabilization parameter $\alpha > 0$
such that 
\bee
\norm{\mathbf{S}(\tau,\sigma) - \mathbf{S}_{r}(\tau,\sigma)}_2 \leq C_{\rm sc}  N^{2/(d-1)}  
e^{-br^{1/(d+1)}} \norm{\mathbf{W}^{\rm st}}_2.
\ee
\end{theorem}

\bpro
Let $B_{R_{\tau}},B_{R_{\sigma}}$ be
bounding boxes for the clusters $\tau$, $\sigma$ satisfying \eqref{eq:admissibility} and
$\Gamma_{\rho} \subset \Gamma$ defined by \eqref{eq:screen}. 
Lemma~\ref{rem:SchurRepresentation} provides a representation for the Schur complement as
\been\label{eq:SchurRepresentation}
\boldsymbol{\phi}^T\left(\mathbf{S}(\tau,\sigma)+\mathbf{D}\right)\boldsymbol{\psi} = 
\skp{W\widetilde{\phi},\psi}_{L^2(\Gamma)}+\alpha\skp{\widetilde{\phi},1}_{L^2(\Gamma)}\skp{\psi,1}_{L^2(\Gamma)}, 
\een
with the following relation between the functions $\psi$, $\widetilde\phi$ and the vectors 
$\boldsymbol{\psi}$, $\boldsymbol{\phi}$, respectively: 
$\psi = \sum_{j=1}^{\abs{\sigma}}\boldsymbol{\psi}_j\chi_{j_{\sigma}}$, 
where the index $j_{\sigma}$ denotes the $j$-th basis function corresponding to the cluster $\sigma$,
and the function $\widetilde{\phi} \in S^{p,1}(\T_h)$ 
is defined by $\widetilde{\phi} = \phi + \phi_{\rho}$ 
with $\phi = \sum_{j=1}^{\abs{\tau}}\boldsymbol{\phi}_j\chi_{j_{\tau}}$
and $\supp \phi_{\rho} \subset \overline{\Gamma_{\rho}}$ such that 
\been\label{eq:SchurOrthogonality}
\skp{W\widetilde{\phi},\widehat{\psi}}_{L^2(\Gamma)}
 = 0 \quad \forall \widehat{\psi} \in S^{p,1}({\mathcal T}_h) \; \text{with}\; 
\supp \widehat{\psi} \subset \overline{\Gamma_{\rho}}.
\een
Our low-rank approximation of the Schur complement matrix $\mathbf{S}(\tau,\sigma)$ will have two
ingredients: first, 
based on the the techniques of Section~\ref{sec:Approximation-solution} we exploit the 
orthogonality \eqref{eq:SchurOrthogonality} to 
construct a low-dimensional space $\widehat W_k$ from which for any $\phi$, the corresponding
function $\widetilde \phi$ can be approximated well. Second,  
we exploit that the function $\psi$ in (\ref{eq:SchurRepresentation}) is 
supported by $\Gamma_{\sigma}$, and we will use Lemma~\ref{lem:lowrankGreensfunction}. 

Let $\delta = \frac{1}{1+\eta}$ and $B_{R_{\sigma}}$, $B_{(1+\delta)R_{\sigma}}$ be concentric boxes.
The symmetry of $W$ leads to
\begin{align}
\label{eq:tmpSchur}
&\skp{W\widetilde{\phi},\psi}_{L^2(\Gamma)}+\alpha\skp{\widetilde{\phi},1}_{L^2(\Gamma)}\skp{\psi,1}_{L^2(\Gamma)} 
= \skp{\widetilde{\phi},W\psi}_{L^2(\Gamma)}+\alpha\skp{\widetilde{\phi},1}_{L^2(\Gamma)}\skp{\psi,1}_{L^2(\Gamma)}
\nonumber \\ 
\qquad &=
\skp{\widetilde{\phi},W\psi}_{L^2( B_{(1+\delta)R_{\sigma}}\cap\Gamma_{\rho} )} + 
\skp{\widetilde{\phi},W\psi}_{L^2(\Gamma \setminus B_{(1+\delta)R_{\sigma}})}  
+\alpha\skp{\widetilde{\phi},1}_{L^2(\Gamma)}\skp{\psi,1}_{L^2(\Gamma)}. 
\end{align}
First, we treat the first term on the right-hand side of \eqref{eq:tmpSchur}.
In view of the symmetry property 
$\mathbf{S}(\tau,\sigma) = \mathbf{S}(\sigma,\tau)^T$, we may assume for approximation
purposes that $\operatorname*{diam} B_{R_\sigma} \leq \operatorname*{diam} B_{R_\tau}$, i.e., 
$\min\{\diam(B_{R_{\tau}}),\diam(B_{R_{\sigma}})\} = \sqrt{d}R_{\sigma}$.Next, the choice of $\delta$  
and the admissibility condition \eqref{eq:admissibility} imply
\bee
\dist(B_{(1+2\delta)R_{\sigma}},B_{R_{\tau}}) \geq \dist(B_{R_{\sigma}},B_{R_{\tau}})-\sqrt{d}\delta R_{\sigma}
\geq \sqrt{d}R_{\sigma}(\eta^{-1}-\delta) > 0.
\ee
Therefore, we have $\widetilde{\phi}|_{B_{(1+2\delta)R_{\sigma}}\cap\Gamma_{\rho}} = \phi_{\rho}|_{B_{(1+2\delta)R_{\sigma}}\cap\Gamma_{\rho}}$
and the orthogonality \eqref{eq:SchurOrthogonality} holds 
on the box $B_{(1+2\delta)R_{\sigma}}$. 
Thus, by definition of $\H_{h,0}$, we have  
$\widetilde{K}\widetilde{\phi} \in \H_{h,0}(B_{(1+2\delta)R_{\sigma}},\Gamma_{\rho},0)$.

As a consequence, Lemma~\ref{cor:lowdimappHS} can be applied to the potential 
$\widetilde{K}\widetilde{\phi}$
with $R := (1+\delta)R_{\sigma}$ and $\kappa := \frac{1}{2+\eta} = \frac{\delta}{1+\delta}$. Note that
$(1+\kappa)(1+\delta) = 1+2\delta$ and $1+\kappa^{-1} = 3+\eta$.
Hence, we get a low dimensional space $\widehat{W}_k$ of dimension
$\dim \widehat{W}_k \leq C_{\rm dim}(3+\eta)^dq^{-d}k^{d+1} =: r$, and
the best approximation $\widehat{\phi} = \Pi_{\widehat{W}_k}\widetilde{\phi}$
to $\widetilde{\phi}$ from the space $\widehat{W}_k$ satisfies
\bee
\norm{\widetilde{\phi}-\widehat{\phi}}_{L^{2}(B_{(1+\delta)R_{\sigma}}\cap\Gamma_{\rho})} \lesssim 
R_{\sigma} h^{-1/2} q^k \triplenorm{\widetilde{K}\widetilde{\phi}}_{h,(1+2\delta)R_{\sigma}} 
\lesssim   h^{-1/2}  e^{-b_1r^{1/(d+1)}}\norm{\widetilde{\phi}}_{H^{1/2}(\Gamma)},
\ee
where we defined $b_1 := -\frac{\ln(q)}{C_{\rm dim}^{1/(d+1)}}q^{d/(d+1)}(3+\eta)^{-d/(1+d)} > 0$ 
to obtain $q^k = e^{-b_1r^{1/(d+1)}}$.
Therefore, we get
\bean\label{eq:Schurtemp1}
\abs{\skp{\widetilde{\phi}-\widehat{\phi},W\psi}_{L^2(B_{(1+\delta)R_{\sigma}}\cap\Gamma_{\rho})}} \lesssim
 h^{-1/2} e^{-b_1r^{1/(d+1)}}\norm{\widetilde{\phi}}_{H^{1/2}(\Gamma)}\norm{W\psi}_{L^{2}(\Gamma)}. 
\eean
The ellipticity of the hyper-singular integral operator on the screen $\Gamma_{\rho} \subsetneq \Gamma$, 
$\supp (\widetilde{\phi} - \phi) = \supp \phi_{\rho} \subset \overline{\Gamma_{\rho}}$,
 and the orthogonality \eqref{eq:SchurOrthogonality} lead to
\bean\label{eq:Schurtemp2}
\norm{\widetilde{\phi}-\phi}_{H^{1/2}(\Gamma)}^2&\lesssim& 
\skp{W(\widetilde{\phi}-\phi),\widetilde{\phi}-\phi}_{L^2(\Gamma)}
= -\skp{W\phi,\widetilde{\phi}-\phi}_{L^2(\Gamma)} \nonumber \\
 &\lesssim& 
\norm{W\phi}_{H^{-1/2}(\Gamma)}\norm{\widetilde{\phi}-\phi}_{H^{1/2}(\Gamma)}
\lesssim \norm{\phi}_{H^{1/2}(\Gamma)}\norm{\widetilde{\phi}-\phi}_{H^{1/2}(\Gamma)}.
\eean
Thus, with the triangle inequality, \eqref{eq:Schurtemp2}, 
the stability of $W:H^{1}(\Gamma)\ra L^{2}(\Gamma)$, and the inverse estimate \eqref{eq:inverse}, we 
can estimate \eqref{eq:Schurtemp1} by 
\bea
\abs{\skp{\widetilde{\phi}-\widehat{\phi},W\psi}_{L^2(B_{(1+\delta)R_{\sigma}}\cap\Gamma_{\rho})}} 
&\lesssim& 
 h^{-1/2} e^{-br^{1/(d+1)}}\left(\norm{\widetilde{\phi}-\phi}_{H^{1/2}(\Gamma)}+
\norm{\phi}_{H^{1/2}(\Gamma)}\right)\norm{W\psi}_{L^{2}(\Gamma)} \\
&\lesssim& 
 h^{-2} e^{-br^{1/(d+1)}}\norm{\phi}_{L^{2}(\Gamma)}\norm{\psi}_{L^2(\Gamma)}. 
\eea
For the second term in \eqref{eq:tmpSchur}, 
we exploit the asymptotic smoothness of the Green's function $G(\cdot,\cdot)$.
First, we mention a standard device in connection with the hyper-singular integral operator, namely, 
it can be represented in terms of the simple-layer operator (see, e.g., \cite[Sec.~6]{Steinbach}): 
\begin{equation}
\label{eq:representation-of-W-in-terms-of-V}
\skp{\widetilde{\phi},W\psi} 
= 
\skp{{\curl}_{\Gamma}\widetilde{\phi},V{\curl}_{\Gamma}\psi},  
\end{equation}
where for a scalar function $v$ defined on $\Gamma$, a lifting operator $\mathcal{L}$, and the outer normal
vector $n$, the surface curl is defined as 
\begin{align*}
{\curl}_\Gamma v &= n \times \gamma_0^{\rm int} (\nabla \mathcal{L} v), \qquad \mbox{ for $d = 3$}, \\
{\curl}_\Gamma v &= n \cdot \gamma_0^{\rm int} (\nabla^T \mathcal{L} v), \quad 
\nabla^T v = (\partial_2 v,-\partial_1 v)^T \qquad \mbox{ for $d = 2$}. 
\end{align*}
The representation (\ref{eq:representation-of-W-in-terms-of-V})
is necessary here, since the kernel of the hyper-singular integral operator is not 
asymptotically smooth on non-smooth surfaces $\Gamma$.

Now, Lemma~\ref{lem:lowrankGreensfunction}
can be applied with $B_Y = B_{R_{\sigma}}$ and $D_X = \Gamma \setminus B_{(1+\delta)R_{\sigma}}$, 
where the choice of $\delta$ implies 
\been
\label{eq:degenerate-approximation-admissibility}
\dist(B_Y,D_X)\geq \frac{1}{2\sqrt{d}(1+\eta)} \diam(B_Y).
\een
Therefore, we get an approximation $G_r(x,y) = \sum_{i=1}^r g_{1,i}(x) g_{2,i}(y)$ such that 
\bean
\label{eq:degenerate-approximation-error}
\norm{G(x,\cdot)-{G}_r(x,\cdot)}_{L^{\infty}(B_{R_{\sigma}})} 
\!&\lesssim&\! \frac{1}{\dist(\{x\},B_{R_\sigma})^{d-2}}e^{-b_2r^{1/d}} 
\quad\! \! \forall x \in \Gamma\setminus B_{(1+\delta)R_{\sigma}};
\eean
here, the constant $b_2>0$ depends only on $d$ and $\eta$. 
As a consequence of \eqref{eq:degenerate-approximation-admissibility} and
\eqref{eq:degenerate-approximation-error}, 
the rank-$r$ operator $W_r$ given by 
\bee
\skp{\widetilde{\phi},W_r\psi}_{L^2(\Gamma\setminus B_{(1+\delta)R_{\sigma}})}:=
\int_{\Gamma\setminus B_{(1+\delta)R_{\sigma}}} \curl_{\Gamma}\widetilde{\phi}(x)
\int_{B_{R_{\sigma}}\cap\Gamma}{G}_r(x,y)\curl_{\Gamma}\psi(y)ds_yds_x \ee 
satisfies with $B:= (\Gamma \setminus B_{(1+\delta)R_{\sigma}}) \times (B_{R_{\sigma}}\cap\Gamma)$
\bea
\abs{\skp{\widetilde{\phi},(W-W_r)\psi}_{L^2(\Gamma \setminus B_{(1+\delta)R_{\sigma}})}} 
&\lesssim& 
\norm{\curl_{\Gamma}\widetilde{\phi}}_{L^2(\Gamma)}
\sqrt{\operatorname*{meas}(\Gamma \cap B_{R_\sigma})}
\norm{G-\widetilde{G}_r}_{L^{\infty}\left(B\right)}
\norm{\curl_{\Gamma}\psi}_{L^2(\Gamma)} \\
&\lesssim& 
h^{-3/2}\delta^{2-d} R_\sigma^{(3-d)/2} e^{-b_2r^{1/d}}
\norm{\widetilde{\phi}}_{H^{1/2}(\Gamma)} \norm{\psi}_{L^2(\Gamma)} \\
&\lesssim& 
h^{-2}e^{-b_2r^{1/d}} \norm{\phi}_{L^2(\Gamma)}
\norm{\psi}_{L^2(\Gamma)},
\eea
where the last two inequalities follow from the inverse estimate Lemma~\ref{lem:inverseinequality}, 
the stability estimate \eqref{eq:Schurtemp2} for the mapping $\phi \mapsto \widetilde \phi$, the assumption 
$d \leq 3$ as well as $R_{\sigma} \leq \eta\diam(\Omega)$, and the choice $\delta = \frac{1}{1+\eta}$. 
Here, the hidden constant additionally depends on $\eta$.

Since the mapping 
\bee
(\phi,\psi)\! \mapsto\! \skp{\widehat{\phi},W\psi}_{L^2(B_{(1+\delta)R_{\sigma}}\cap\Gamma_{\rho})} + 
\skp{\widetilde{\phi},W_r\psi}_{L^2(\Gamma \setminus B_{(1+\delta)R_{\sigma}})}
\ee
defines a bounded bilinear form on $L^2(\Gamma)$,
there exists a linear operator $\widehat{W}_r:L^2(\Gamma)\ra L^2(\Gamma)$ such that
\bee
\skp{\widehat{\phi},W\psi}_{L^2(B_{(1+\delta)R_{\sigma}}\cap\Gamma_{\rho})} + 
\skp{\widetilde{\phi},W_r\psi}_{L^2(\Gamma \setminus B_{(1+\delta)R_{\sigma}})}
= \skp{\widehat{W}_r\phi,\psi}_{L^2(\Gamma)},
\ee
and the dimension of the range of $\widehat{W}_r$ is bounded by $2r$.

Therefore, we get 
\bea
\abs{\skp{W\widetilde{\phi},\psi}_{L^2(\Gamma)}
 - \skp{\widehat{W}_r\phi,\psi}_{L^2(\Gamma)}}     
\lesssim h^{-2}e^{-br^{1/(d+1)}}\norm{\phi}_{L^2(\Gamma)}
\norm{\psi}_{L^2(\Gamma)},
\eea
with $b := \min\{b_1,b_2\}$.
This leads to a
matrix $\widehat{\mathbf{S}_{r}}(\tau,\sigma)$ of rank $2r+1$ such that
\bee
\norm{\mathbf{S}(\tau,\sigma)+\mathbf{D}-\widehat{\mathbf{S}_{r}}(\tau,\sigma)}_2 = 
\sup_{\boldsymbol{\phi}\in\R^{\abs{\tau}},\boldsymbol{\psi}\in \R^{\abs{\sigma}}} 
\frac{\abs{\boldsymbol{\phi}^T(\mathbf{S}(\tau,\sigma)+\mathbf{D}-\widehat{\mathbf{S}_{r}}(\tau,\sigma))\boldsymbol{\psi}}}
{\norm{\boldsymbol{\phi}}_2\norm{\boldsymbol{\psi}}_2} 
\leq C h^{d-3} e^{-br^{1/(d+1)}},
\ee
where we have used \eqref{eq:basisisomorphism}.
Consequently we can find a matrix $\mathbf{S}_{r}(\tau,\sigma) := \widehat{\mathbf{S}_{r}}(\tau,\sigma)-\mathbf{D}$ 
of rank $2r+3$ such that
\bee
\norm{\mathbf{S}(\tau,\sigma)-\mathbf{S}_{r}(\tau,\sigma)}_2  
\leq C  h^{d-3} e^{-br^{1/(d+1)}}.
\ee
The estimate $\frac{1}{\norm{\mathbf{W}^{\rm st}}_2}\lesssim h^{-d+1}$ 
(with implied constant depending on $\alpha$) from \cite[Lemma~12.9]{Steinbach}
and $h \simeq N^{-1/(d-1)}$ finish the proof.
\epro

\subsection{Existence of $\H$-Cholesky decomposition}

In this subsection, we will use the approximation of the Schur complement from the previous section
to prove the existence of an (approximate) $\H$-Cholesky decomposition. 
We start with a hierarchical relation of the Schur complements $\mathbf{S}(\tau,\tau)$. \newline 

The Schur complements $\mathbf{S}(\tau,\tau)$ for a block $\tau \in \mathbb{T}_{\mathcal{I}}$ 
can be derived from the Schur complements of its sons $\tau_1$, $\tau_2$ by 
\bee
\mathbf{S}(\tau,\tau) = \begin{pmatrix} \mathbf{S}(\tau_1,\tau_1) & \mathbf{S}(\tau_1,\tau_2) \\ 
\mathbf{S}(\tau_2,\tau_1) & \mathbf{S}(\tau_2,\tau_2) + \mathbf{S}(\tau_2,\tau_1)\mathbf{S}(\tau_1,\tau_1)^{-1}\mathbf{S}(\tau_1,\tau_2) \end{pmatrix},
\ee
A proof of this relation can be found in \cite[Lemma 3.1]{Bebendorf07}. One should note that
the proof does not use any properties of the matrix $\mathbf{W}^{\rm st}$ other than invertibility and 
existence of a Cholesky decomposition. 
Moreover, we have by definition of $\mathbf{S}(\tau,\tau)$ that $\mathbf{S}(\mathcal{I},\mathcal{I}) = \mathbf{W}^{\rm st}$.

If $\tau$ is a leaf, we get the Cholesky decomposition of $\mathbf{S}(\tau,\tau)$ by the classical 
Cholesky decomposition,
which exists since $\mathbf{W}^{\rm st}$ has a Cholesky decomposition.
If $\tau$ is not a leaf, we use the hierarchical relation of the Schur complements to
define a Cholesky decomposition of the Schur complement
 $\mathbf{S}(\tau,\tau)$ by 
\been\label{eq:LUdefinition}
\mathbf{C}(\tau) := \begin{pmatrix} \mathbf{C}(\tau_1) & 0 \\ \mathbf{S}(\tau_2,\tau_1)(\mathbf{C}(\tau_1)^T)^{-1} & \mathbf{C}(\tau_2)  \end{pmatrix}, \quad 
\een
with $\mathbf{S}(\tau_1,\tau_1) = \mathbf{C}(\tau_1)\mathbf{C}(\tau_1)^T$,
 $\mathbf{S}(\tau_2,\tau_2) = \mathbf{C}(\tau_2)\mathbf{C}(\tau_2)^T$ and indeed get 
$\mathbf{S}(\tau,\tau) = \mathbf{C}(\tau) \mathbf{C}(\tau)^T$.
Moreover, the uniqueness of the Cholesky decomposition of $\mathbf{W}^{\rm st}$ implies that due to
$\mathbf{C}\mathbf{C}^T = \mathbf{W}^{\rm st} = \mathbf{S}(\mathcal{I},\mathcal{I}) = \mathbf{C}(\mathcal{I})\mathbf{C}(\mathcal{I})^T$, we have
$\mathbf{C} = \mathbf{C}(\mathcal{I})$.

The existence of the inverse $\mathbf{C}(\tau_1)^{-1}$
follows from the representation \eqref{eq:LUdefinition}
by induction over the levels, since on a leaf the existence is clear and the matrices 
$\mathbf{C}(\tau)$ are block triangular matrices. Consequently, the inverse of
$\mathbf{S}(\tau,\tau)$ exists. 

Moreover, as shown in \cite[Lemma~{22}]{GrasedyckKriemannLeBorne}  in the context of $LU$-factorizations
instead of Cholesky decompositions, the restriction of the lower triangular part 
$\mathbf{S}(\tau_2,\tau_1)(\mathbf{C}(\tau_1)^T)^{-1}$ 
of the matrix $\mathbf{C}(\tau)$ to a subblock $\tau_2'\times\tau_1'$
with $\tau_i'$ a son of $\tau_i$ satisfies
\been
\label{eq:foo}
\left(\mathbf{S}(\tau_2,\tau_1)(\mathbf{C}(\tau_1)^T)^{-1}\right)|_{\tau_2'\times\tau_1'} = 
\mathbf{S}(\tau_2',\tau_1')(\mathbf{C}(\tau_1')^T)^{-1}.
\een

The following lemma shows that the spectral norm of the inverse
$\mathbf{C}(\tau)^{-1}$ can be bounded by the norm of the inverse
$\mathbf{C}(\mathcal{I})^{-1}$.

\begin{lemma}\label{lem:LUnorm}
For $\tau\in \mathbb{T}_{\mathcal{I}}$,
let $\mathbf{C}(\tau)$ be given by \eqref{eq:LUdefinition}. Then,
\bea
\max_{\tau\in\mathbb{T}_{\mathcal{I}}}\norm{\mathbf{C}(\tau)^{-1}}_2 &=& \norm{\mathbf{C}(\mathcal{I})^{-1}}_2, \\
\eea
\end{lemma}
\begin{proof}
With the block structure of \eqref{eq:LUdefinition}, we get the inverse
\bee
\mathbf{C}(\tau)^{-1} = \begin{pmatrix} \mathbf{C}(\tau_1)^{-1} & 0 \\ -\mathbf{C}(\tau_2)^{-1} \mathbf{S}(\tau_2,\tau_1)(\mathbf{C}(\tau_1)^T)^{-1}\mathbf{C}(\tau_1)^{-1} & \mathbf{C}(\tau_2)^{-1} \end{pmatrix}.
\ee
So, we get by choosing $\mathbf{x}$ such that $\mathbf{x}_i = 0$ for $i \in \tau_1$ that
\bea
\norm{\mathbf{C}(\tau)^{-1}}_2 = \sup_{\mathbf{x}\in \R^{\abs{\tau}},\norm{x}_2=1}\norm{\mathbf{C}(\tau)^{-1}\mathbf{x}}_2 
\geq \sup_{\mathbf{x}\in \R^{\abs{\tau_2}},\norm{x}_2=1}\norm{\mathbf{C}(\tau_2)^{-1}\mathbf{x}}_2 
= \norm{\mathbf{C}(\tau_2)^{-1}}_2.
\eea
The same argument for $\left(\mathbf{C}(\tau)^{-1}\right)^T$ leads to 
\bea
\norm{\mathbf{C}(\tau)^{-1}}_2 = \norm{\left(\mathbf{C}(\tau)^{-1}\right)^T}_2 \geq \norm{\mathbf{C}(\tau_1)^{-1}}_2.
\eea
Thus, we have $\norm{\mathbf{C}(\tau)^{-1}}_2 \geq \max_{i=1,2}\norm{\mathbf{C}(\tau_i)^{-1}}_2$ and as a consequence
$\max_{\tau\in\mathbb{T}_{\mathcal{I}}}\norm{\mathbf{C}(\tau)^{-1}}_2 = \norm{\mathbf{C}(\mathcal{I})^{-1}}_2$.
\end{proof}


We are now in position to prove Theorem~\ref{th:HLU}:

\begin{proof}[of Theorem~\ref{th:HLU}]
{\em Proof of (\ref{item:th:HLU-i}):} 
In the following, we show that every admissible subblock $\tau\times\sigma$ of $\mathbf{C}(\mathcal{I})$, 
recursively defined by \eqref{eq:LUdefinition}, has a rank-$r$ approximation. Since an admissible block
of the lower triangular part of $\mathbf{C}(\mathcal{I})$
has to be a subblock of a matrix $\mathbf{C}(\tau')$ for some 
$\tau' \in \mathbb{T}_{\mathcal{I}}$, we get in view of \eqref{eq:foo} that 
$\mathbf{C}(\mathcal{I})|_{\tau\times\sigma} = \mathbf{S}(\tau,\sigma)(\mathbf{C}(\sigma)^T)^{-1}$.
Theorem~\ref{lem:Schur} provides a rank-$r$ approximation 
${\mathbf S}_{r}(\tau,\sigma)$ to ${\mathbf S}(\tau,\sigma)$. Therefore, we can estimate 
\bea
\norm{\mathbf{C}(\mathcal{I})|_{\tau\times\sigma}\! - \! \mathbf{S}_{r}(\tau,\sigma)(\mathbf{C}(\sigma)^T)^{-1}}_2 \!\! &=& \!\!
\norm{\left(\mathbf{S}(\tau,\sigma)-\mathbf{S}_{r}(\tau,\sigma)\right)(\mathbf{C}(\sigma)^T)^{-1}}_2 \\
&\leq&\!\! C_{\rm sc}   N^{2/(d-1)}  e^{-br^{1/(d+1)}}\norm{(\mathbf{C}(\sigma')^T)^{-1}}_2\norm{\mathbf{W}^{\rm st}}_2.
\eea 
Since $\mathbf{S}_{r}(\tau,\sigma)(\mathbf{C}(\sigma)^T)^{-1}$ is a rank-$r$ matrix for each $\eta$-admissible 
cluster pair $(\tau,\sigma)$, we immediately get an $\H$-matrix approximation $\mathbf{C}_{\H}$
of the Cholesky factor $\mathbf{C}(\mathcal{I}) = \mathbf{C}$. With Lemma~\ref{lem:spectralnorm} 
and Lemma~\ref{lem:LUnorm}, we get
\bee
\norm{\mathbf{C}-\mathbf{C}_{\H}}_2\leq C_{\rm sc} C_{\rm sp}  N^{2/(d-1)} 
{\rm depth}(\mathbb{T}_{\mathcal{I}})e^{-br^{1/(d+1)}}
\norm{\mathbf{C}^{-1}}_2\norm{\mathbf{W}^{\rm st}}_2,
\ee
and with $\norm{\mathbf{W}^{\rm st}}_2 = \norm{\mathbf{C}}_2^2$, we conclude the proof of 
(\ref{item:th:HLU-i}). 

{\em Proof of (\ref{item:th:HLU-ii}):} 
Since $\mathbf{W}^{\rm st}=\mathbf{C}\mathbf{C}^T$, the triangle inequality leads to
\bea
\norm{\mathbf{W}^{\rm st}-\mathbf{C}_{\H}\mathbf{C}_{\H}^T}_2 &\leq& 
\norm{\mathbf{C}-\mathbf{C}_{\H}}_2\norm{\mathbf{C}^T}_2 + 
\norm{\mathbf{C}^T-\mathbf{C}_{\H}^T}_2 \norm{\mathbf{C}}_2 + 
\norm{\mathbf{C}-\mathbf{C}_{\H}}_2\norm{\mathbf{C}^T-\mathbf{C}_{\H}^T}_2 \\
&\leq& 2C_{\rm sc}C_{\rm sp}\kappa_2(\mathbf{C}){\rm depth}(\mathbb{T}_{\mathcal{I}})
 N^{2/(d-1)}  e^{-br^{1/(d+1)}}\norm{\mathbf{W}^{\rm st}}_2\\
 & & +\kappa_2(\mathbf{C})^2C_{\rm sc}^2C_{\rm sp}^2 {\rm depth}(\mathbb{T}_{\mathcal{I}})^2
 N^{4/(d-1)}  e^{-2br^{1/(d+1)}}\frac{\norm{\mathbf{W}^{\rm st}}_2^2}{\norm{\mathbf{C}}_2^2},
\eea
and the equality $\kappa_2(\mathbf{W}^{\rm st}) = \kappa_2(\mathbf{C})^2$ finishes the proof of 
(\ref{item:th:HLU-ii}).

{\em Proof of (\ref{item:th:HLU-iii}):}
The approximate $LU$-factors $\mathbf{L_{\H}}, \mathbf{U_{\H}}$ can be constructed from $\mathbf{C_{\H}}$ by
\been
\mathbf{L}_{\H}\mathbf{U}_{\H} = \bpm \mathbf{C_{\H}} & 0 \\ \boldsymbol{\ell}^T & -\abs{\mathbf{B}}^2 \epm
 \bpm \mathbf{C_{\H}}^T & \boldsymbol{\ell} \\ 0 & 1 \epm = 
\bpm \mathbf{C_{\H}}\mathbf{C_{\H}}^T & \mathbf{B} \\ \mathbf{B}^T & 0 \epm,
\een
where $\boldsymbol{\ell} \in \R^N$ solves $\mathbf{C}_{\H}\boldsymbol{\ell} = \mathbf{B}$, and the error estimate 
follows from (\ref{item:th:HLU-ii}).
\end{proof}

\section{Numerical Examples}
\label{sec:numerics}
In this section, we present some numerical examples in dimension $d = 3$  to illustrate
the theoretical estimates derived in the previous sections.
Further numerical examples about $\H$-matrix approximation of inverse BEM matrices and black-box
preconditioning with an $\H$-LU decomposition can be found, e.g., in 
\cite{GrasedyckDissertation,Bebendorf05,Grasedyck05,BoermBuch,FMPBEM}, where the focus is, however,
on the weakly-singular integral operator.

With the choice $\eta=2$ for the admissibility parameter in \eqref{eq:admissibility},
the clustering is done by the standard geometric clustering algorithm,
i.e., by choosing axis-parallel bounding boxes of minimal volume and
splitting these bounding boxes in half across the largest face 
until they are admissible or contain less degrees of freedom than $n_{\text{leaf}}$,
which we choose as $n_{\text{leaf}} = 50$ for our computations.
An approximation to the inverse Galerkin matrix is computed by using the C++-software package 
BEM++ \cite{BEMpp}. The $\H$-matrices are assembled using ACA and the C++-library
AHMED \cite{AHMED}.

Our numerical experiments are performed for the Galerkin discretization of 
the stabilized hyper-singular integral operator $\mathbf{W}^{\rm st}$ 
as described in Section~\ref{sec:stabGalerkin} with $\alpha = 1$.
The geometry is the crankshaft generated by NETGEN \cite{netgen} visualized in Figure~\ref{fig:crankshaft}.
We employ a fixed triangulation of the crankshaft consisting of $5,393$ nodes and $6,992$ elements. 
\begin{figure}[h]
\begin{center}
\includegraphics[width=0.35\textwidth]{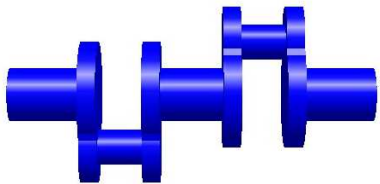}
\caption{\footnotesize Crankshaft domain}
\label{fig:crankshaft}
\end{center}
\end{figure}
\begin{example} 
{\rm 
The numerical calculations are performed for the polynomial degree $p=2$, resulting in 
$N = 13,986$ degrees of freedom. The largest block of $\mathbf{W}_{\H}$ has a size of $1,746^2$. 
In Figure~\ref{fig:3DBEMp2r}, we compare the decrease of the upper bound 
$\norm{\mathbf{I}-\mathbf{W}^{\rm st}\mathbf{W}_{\H}}_2$ of the relative error with the increase 
of the block-rank. Figure~\ref{fig:3DBEMp2Mem} shows the storage requirement for the 
computed $\H$-matrix approximation in MB. Storing the dense matrix would need $1,492$ MB.
We observe exponential convergence in the block rank, even with a convergence behavior 
$\exp(-br^{1/2})$, which is faster than the rate of $\exp(-br^{1/4})$ guaranteed by Theorem~\ref{th:Happrox}.
Moreover, we also observe exponential convergence of the error compared to the increase of required memory. 
\begin{figure}[hbt]
\begin{minipage}{.50\linewidth}
\centering
\psfrag{Error}[c][c]{%
 \footnotesize  Error}
\psfrag{Block rank r }[c][c]{%
 \footnotesize  Block rank r}
\psfrag{asd}[l][l]{\tiny $\exp(-2.2\, r^{1/2})$}
\psfrag{jkl}[l][l]{\tiny $\norm{I-\mathbf{W}^{\rm st}W_{\H}}_2$}
\includegraphics[width=0.80\textwidth]{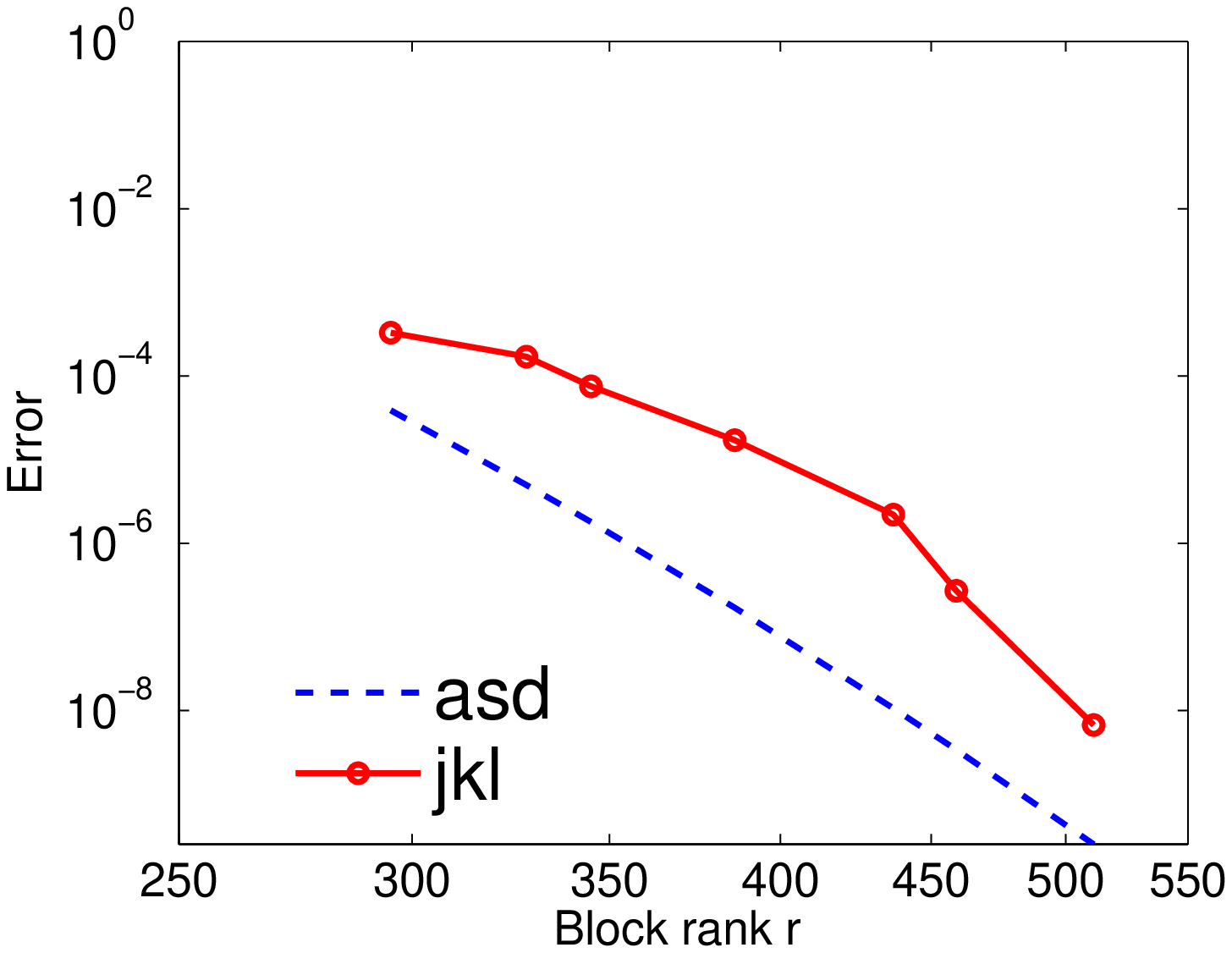}
\caption{\footnotesize Exponential convergence in block rank}
\label{fig:3DBEMp2r}
\end{minipage}
\begin{minipage}{.50\linewidth}
\centering
\psfrag{Error}[c][c]{%
 \footnotesize  Error}
\psfrag{Memory (MB)}[c][c]{%
 \footnotesize  Memory (MB)}
\psfrag{asd}[l][l]{\tiny $\exp(-1.6\, r^{1/2})$}
\psfrag{jkl}[l][l]{\tiny $\norm{I-\mathbf{W}^{\rm st}W_{\H}}_2$}
\includegraphics[width=0.80\textwidth]{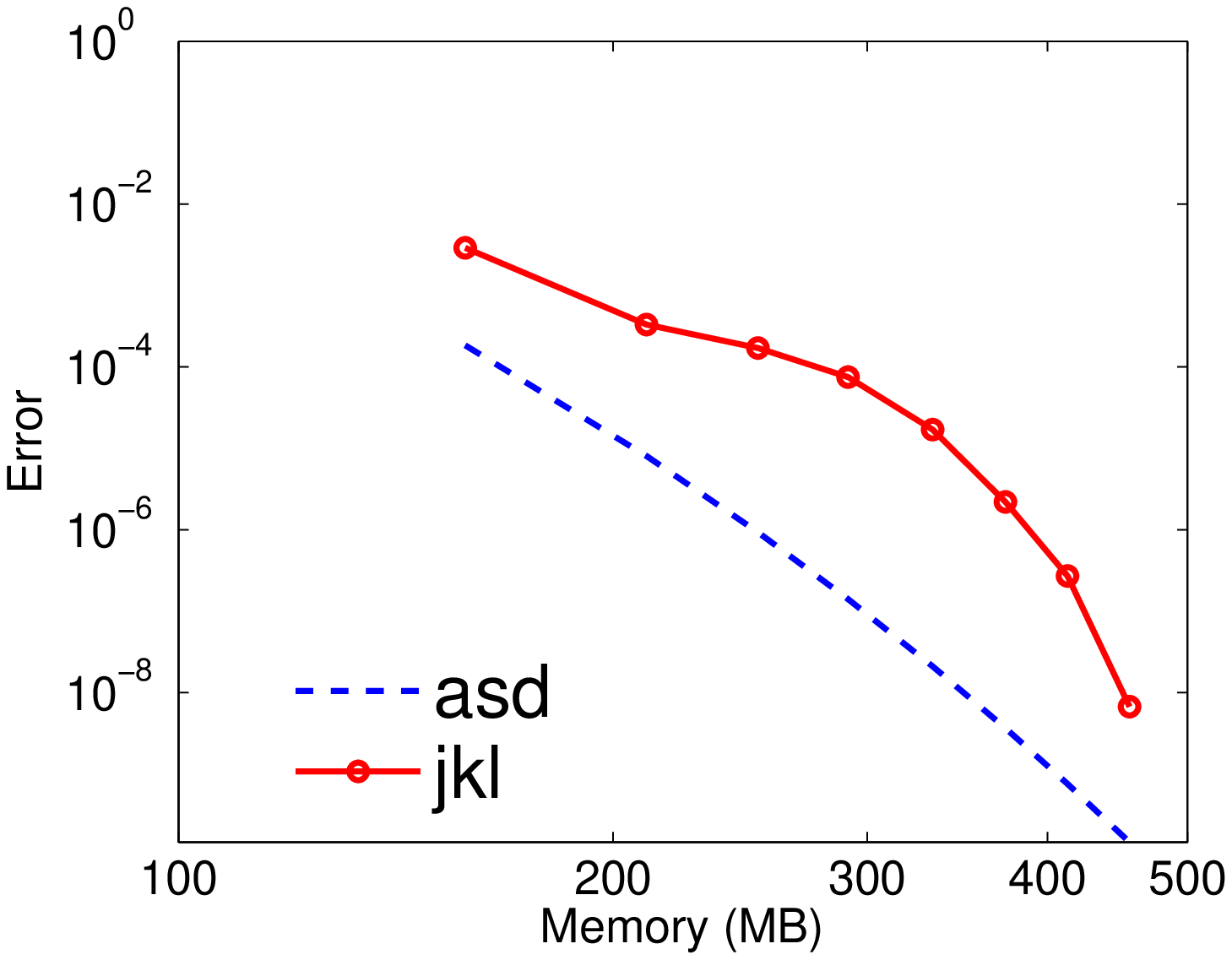}
\caption{\footnotesize Exponential convergence in memory required}
\label{fig:3DBEMp2Mem}
\end{minipage}
\end{figure}
\eex
}
\end{example}
\begin{example}
{\rm 
We consider the case $p = 3$, which leads to $N= 31,466$ degrees of freedom. 
The largest block of $\mathbf{W}_{\H}$ has a size of $3,933^2$. 
Storing the dense matrix would need $7,608$ MB.
%
\begin{figure}[hbt]
\begin{minipage}{.50\linewidth}
\centering
\psfrag{Error}[c][c]{%
 \footnotesize  Error}
\psfrag{Block rank r }[c][c]{%
 \footnotesize  Block rank r}
\psfrag{asd}[l][l]{\tiny $\exp(-1.7\, r^{1/2})$}
\psfrag{jkl}[l][l]{\tiny $\norm{I-\mathbf{W}^{\rm st}W_{\H}}_2$}
\includegraphics[width=0.80\textwidth]{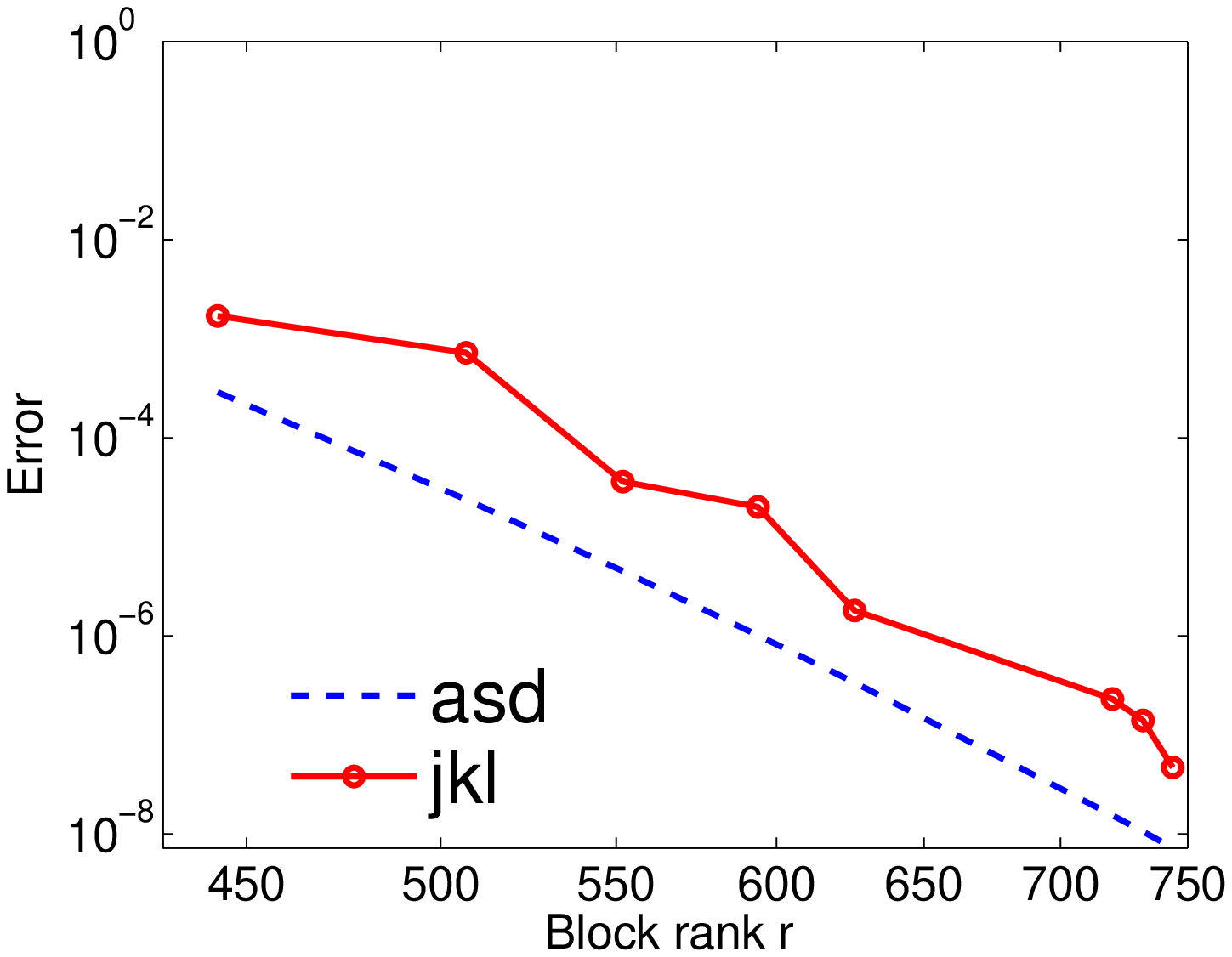}
\caption{\footnotesize Exponential convergence in block rank}
\label{fig:3DBEMp3r}
\end{minipage}
\begin{minipage}{.50\linewidth}
\centering
\psfrag{Error}[c][c]{%
 \footnotesize  Error}
\psfrag{Memory (MB)}[c][c]{%
 \footnotesize  Memory (MB)}
\psfrag{asd}[l][l]{\tiny $\exp(-0.8\, r^{1/2})$}
\psfrag{jkl}[l][l]{\tiny $\norm{I-\mathbf{W}^{\rm st} W_{\H}}_2$}
\includegraphics[width=0.80\textwidth]{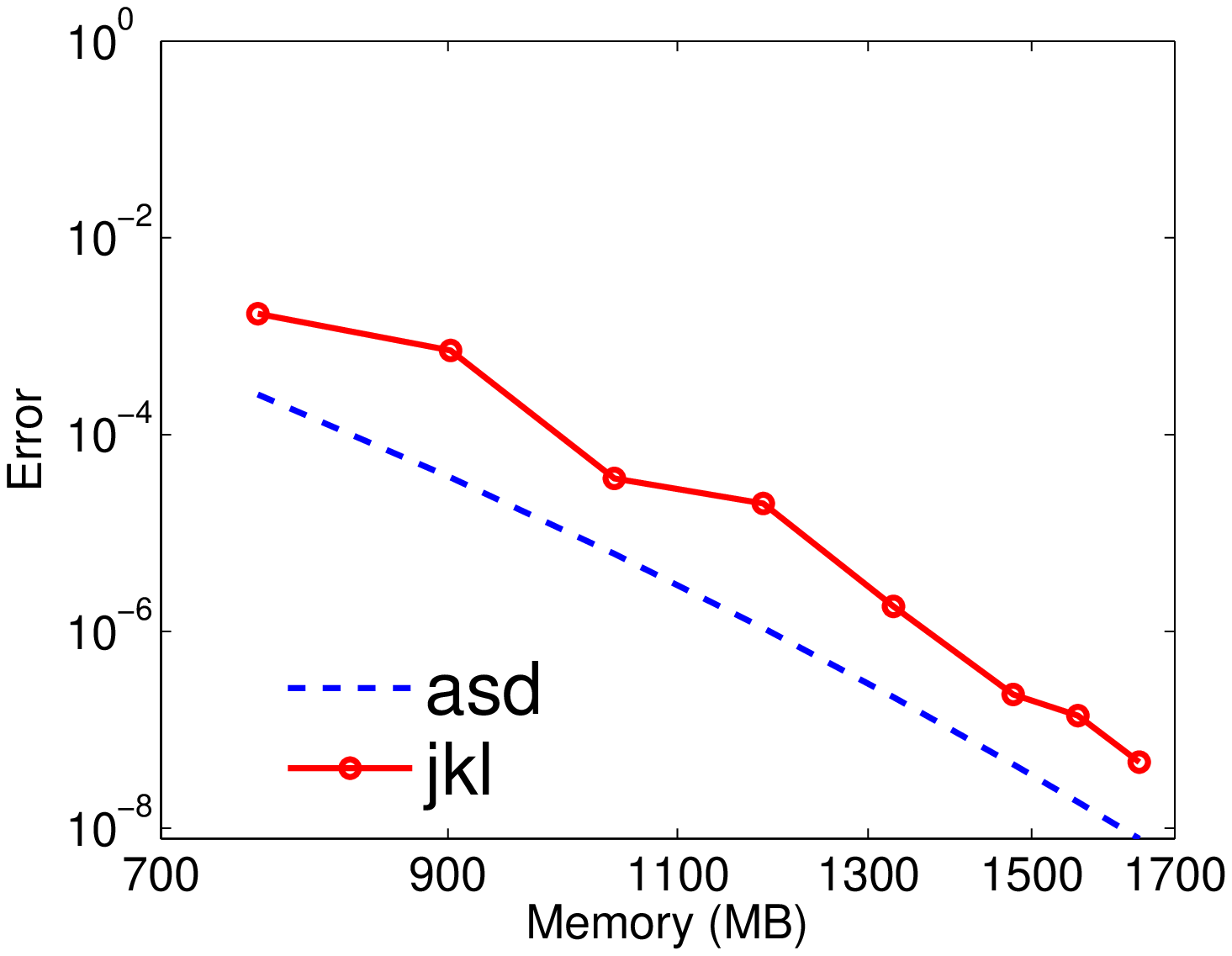}
\caption{\footnotesize Exponential convergence in memory required}
\label{fig:3DBEMp3Mem}
\end{minipage}
\end{figure}
We observe in Figure~\ref{fig:3DBEMp3r} exponential convergence both in the block rank and in the memory.
}\eex
\end{example}

\nocite{*}
\bibliography{bibliography_2}{}

\newcommand{\etalchar}[1]{$^{#1}$}
\providecommand{\bysame}{\leavevmode\hbox to3em{\hrulefill}\thinspace}
\providecommand{\MR}{\relax\ifhmode\unskip\space\fi MR }
\providecommand{\MRhref}[2]{%
  \href{http://www.ams.org/mathscinet-getitem?mr=#1}{#2}
}
\providecommand{\href}[2]{#2}
\begin{thebibliography}{DKP{\etalchar{+}}08}

\bibitem[Beb00]{bebendorf00}
M.~Bebendorf, \emph{Approximation of boundary element matrices}, Numer. Math.
  \textbf{86} (2000), no.~4, 565--589.

\bibitem[Beb05a]{Bebendorf4}
\bysame, \emph{Efficient inversion of {G}alerkin matrices of general
  second-order elliptic differential operators with nonsmooth coefficients},
  Math. Comp. \textbf{74} (2005), 1179--1199.

\bibitem[Beb05b]{Bebendorf05}
\bysame, \emph{Hierarchical {LU} decomposition-based preconditioners for
  {BEM}}, Computing \textbf{74} (2005), no.~3, 225--247.

\bibitem[Beb07]{Bebendorf07}
\bysame, \emph{Why finite element discretizations can be factored by triangular
  hierarchical matrices}, SIAM J. Numer. Anal. \textbf{45} (2007), no.~4,
  1472--1494.

\bibitem[Beb12]{AHMED}
\bysame, \emph{Another software library on hierarchical matrices for elliptic
  differential equations ({AHMED})},
  http://bebendorf.ins.uni-bonn.de/AHMED.html (2012).

\bibitem[BG05]{boerm-grasedyck05}
S.~B{\"o}rm and L.~Grasedyck, \emph{Hybrid cross approximation of integral
  operators}, Numer. Math. \textbf{101} (2005), no.~2, 221--249.

\bibitem[BH03]{Bebendorf}
M.~Bebendorf and W.~Hackbusch, \emph{Existence of {$\mathcal{H}$}-matrix
  approximants to the inverse {FE}-matrix of elliptic operators with
  {$L^{\infty}$}-coefficients}, Numer. Math. \textbf{95} (2003), no.~1, 1--28.

\bibitem[B{\"o}r10a]{Boerm}
S.~B{\"o}rm, \emph{Approximation of solution operators of elliptic partial
  differential equations by {$\mathcal{H}$}- and {$\mathcal{H}^2$}-matrices},
  Numer. Math. \textbf{115} (2010), no.~2, 165--193.

\bibitem[B{\"o}r10b]{BoermBuch}
\bysame, \emph{Efficient numerical methods for non-local operators}, EMS Tracts
  in Mathematics, vol.~14, European Mathematical Society (EMS), Z\"urich, 2010.

\bibitem[BS02]{BrennerScott}
S.~C. Brenner and L.~R. Scott, \emph{The mathematical theory of finite element
  methods}, Texts in Applied Mathematics, vol.~15, Springer-Verlag, New York,
  2002.

\bibitem[DFG{\etalchar{+}}01]{DFGHS}
W.~Dahmen, B.~Faermann, I.~G. Graham, W.~Hackbusch, and S.~A. Sauter,
  \emph{Inverse inequalities on non-quasiuniform meshes and application to the
  mortar element method}, Math. Comp. \textbf{73} (2001), 1107--1138.

\bibitem[DKP{\etalchar{+}}08]{demkowicz-kurtz-pardo-paszynski-rachowicz-zdunek%
08}
Leszek Demkowicz, Jason Kurtz, David Pardo, Maciej Paszy{\'n}ski, Waldemar
  Rachowicz, and Adam Zdunek, \emph{Computing with {$hp$}-adaptive finite
  elements. {V}ol. 2}, Chapman \& Hall/CRC Applied Mathematics and Nonlinear
  Science Series, Chapman \& Hall/CRC, Boca Raton, FL, 2008, Frontiers: {T}hree
  dimensional elliptic and Maxwell problems with applications.

\bibitem[FMP12]{FMPIcosahom}
M.~Faustmann, J.~M. Melenk, and D.~Praetorius, \emph{A new proof for existence
  of {$\mathcal{H}$}-matrix approximants to the inverse of {FEM} matrices: the
  {D}irichlet problem for the {L}aplacian}, ASC Report 51/2012, Institute for
  Analysis and Scientific Computing, Vienna University of Technology, Wien
  (2012).

\bibitem[FMP13a]{FMPBEM}
\bysame, \emph{Existence of $\mathcal{H}$-matrix approximants to the inverse of
  {BEM} matrices: the simple-layer operator}, ASC Report 20/2013, Institute for
  Analysis and Scientific Computing, Vienna University of Technology, Wien
  (2013).

\bibitem[FMP13b]{FMPFEM}
\bysame, \emph{$\mathcal{H}$-matrix approximability of the inverse of {FEM}
  matrices}, ASC Report 37/2013, Institute for Analysis and Scientific
  Computing, Vienna University of Technology, Wien (2013).

\bibitem[GH03]{GrasedyckHackbusch}
L.~Grasedyck and W.~Hackbusch, \emph{Construction and arithmetics of
  {$\mathcal{H}$}-matrices}, Computing \textbf{70} (2003), no.~4, 295--334.

\bibitem[GHS05]{GHS}
I.~G. Graham, W.~Hackbusch, and S.~A. Sauter, \emph{Finite elements on
  degenerate meshes: inverse-type inequalities and applications}, IMA J. Numer.
  Anal. \textbf{25} (2005), no.~2, 379--407.

\bibitem[GKLB09]{GrasedyckKriemannLeBorne}
L.~Grasedyck, R.~Kriemann, and S.~Le~Borne, \emph{Domain decomposition based
  {$\mathcal{H}$}-{LU} preconditioning}, Numer. Math. \textbf{112} (2009),
  no.~4, 565--600.

\bibitem[GR97]{greengard-rokhlin97}
L.~Greengard and V.~Rokhlin, \emph{A new version of the fast multipole method
  for the {L}aplace in three dimensions}, Acta Numerica 1997, Cambridge
  University Press, 1997, pp.~229--269.

\bibitem[Gra01]{GrasedyckDissertation}
L.~Grasedyck, \emph{Theorie und {A}nwendungen {H}ierarchischer {M}atrizen},
  Ph.D. thesis, Universit{\"a}t Kiel, 2001.

\bibitem[Gra05]{Grasedyck05}
\bysame, \emph{Adaptive recompression of {$\mathcal{H}$}-matrices for {BEM}},
  Computing \textbf{74} (2005), no.~3, 205--223.

\bibitem[Hac99]{Hackbusch99}
W.~Hackbusch, \emph{A sparse matrix arithmetic based on
  {$\mathcal{H}$}-matrices. {I}ntroduction to {$\mathcal{H}$}-matrices},
  Computing \textbf{62} (1999), no.~2, 89--108.

\bibitem[Hac09]{HackbuschBuch}
\bysame, \emph{Hierarchische {M}atrizen: {A}lgorithmen und {A}nalysis},
  Springer, 2009.

\bibitem[HJ13]{horn-johnson13}
R.A. Horn and Ch.R. Johnson, \emph{Matrix analysis}, second ed., Cambridge
  University Press, Cambridge, 2013.

\bibitem[HKS00]{HackbuschKhoromskijSauter}
W.~Hackbusch, B.~Khoromskij, and S.~A. Sauter, \emph{On
  {$\mathcal{H}^2$}-matrices}, Lectures on Applied Mathematics (2000), 9--29.

\bibitem[HN89]{hackbusch-nowak89}
W.~Hackbusch and Z.P. Nowak, \emph{On the fast matrix multiplication in the
  boundary element method by panel clustering}, Numer. Math. \textbf{54}
  (1989), 463--491.

\bibitem[HS93]{hackbusch-sauter93}
W.~Hackbusch and S.A. Sauter, \emph{On the efficient use of the {G}alerkin
  method to solve {F}redholm integral equations}, Proceedings of {ISNA}
  '92---{I}nternational {S}ymposium on {N}umerical {A}nalysis, {P}art {I}
  ({P}rague, 1992), vol.~38, 1993, pp.~301--322.

\bibitem[KS99]{karniadakis-sherwin99}
G.E. Karniadakis and S.J. Sherwin, \emph{Spectral/hp element methods for cfd},
  Oxford University Press, 1999.

\bibitem[Ne{\v{c}}67]{Necas}
J.~Ne{\v{c}}as, \emph{Les m\'ethodes directes en th\'eorie des \'equations
  elliptiques}, Masson et Cie, \'Editeurs, Paris, 1967.

\bibitem[NH88]{hackbusch-nowak88}
Z.P. Novak and W.~Hackbusch, \emph{Complexity of the method of panels},
  Computational processes and systems, {N}o.\ 6 ({R}ussian), ``Nauka'', Moscow,
  1988, pp.~233--244.

\bibitem[NS74]{nitsche-schatz74}
Joachim~A. Nitsche and Alfred~H. Schatz, \emph{Interior estimates for
  {R}itz-{G}alerkin methods}, Math. Comp. \textbf{28} (1974), 937--958.

\bibitem[Rat98]{rathsfeld98}
A.~Rathsfeld, \emph{A wavelet algorithm for the boundary element solution of a
  geodetic boundary value problem}, Comput. Methods Appl. Mech. Engrg.
  \textbf{157} (1998), no.~3-4, 267--287, Seventh Conference on Numerical
  Methods and Computational Mechanics in Science and Engineering (NMCM 96)
  (Miskolc).

\bibitem[Rat01]{rathsfeld01}
\bysame, \emph{On a hierarchical three-point basis in the space of piecewise
  linear functions over smooth surfaces}, Problems and methods in mathematical
  physics ({C}hemnitz, 1999), Oper. Theory Adv. Appl., vol. 121, Birkh\"auser,
  Basel, 2001, pp.~442--470.

\bibitem[Rok85]{rokhlin85}
V.~Rokhlin, \emph{Rapid solution of integral equations of classical potential
  theory}, J. Comput. Phys. \textbf{60} (1985), 187--207.

\bibitem[Sau92]{sauter92}
S.A. Sauter, \emph{{\"U}ber die effiziente {V}erwendung des
  {G}alerkinverfahrens zur {L\"o}sung {F}redholmscher {I}ntegralgleichungen},
  Ph.D. thesis, {U}niversit\"at {K}iel, 1992.

\bibitem[SBA{\etalchar{+}}15]{BEMpp}
W.~Smigaj, T.~Betcke, S.~R. Arridge, J.~Phillips, and M.~Schweiger,
  \emph{Solving boundary integral problems with {BEM}++}, ACM Transactions on
  Mathematical Software (to appear (2015)).

\bibitem[Sch97]{netgen}
J.~Sch{\"o}berl, \emph{{NETGEN} - {A}n advancing front 2{D}/3{D}-mesh generator
  based on abstract rules}, Comput.Visual.Sci (1997), no.~1, 41--52.

\bibitem[Sch98a]{schneider98}
R.~Schneider, \emph{Multiskalen- und {W}avelet-{M}atrixkompression:
  {A}nalysisbasierte {M}ethoden zur effizienten {L}\"osung gro\ss{}er
  vollbesetzter {G}leichungssysteme}, Advances in Numerical Mathematics,
  Teubner, 1998.

\bibitem[Sch98b]{SchwabBuch}
Ch. Schwab, \emph{{$p$}- and {$hp$}-finite element methods}, Numerical
  Mathematics and Scientific Computation, The Clarendon Press Oxford University
  Press, New York, 1998, Theory and applications in solid and fluid mechanics.

\bibitem[Sch06]{schreittmiller06}
Robert Schrittmiller, \emph{Zur {A}pproximation der {L}{\"o}sungen elliptischer
  {S}ysteme partieller {D}ifferentialgleichungen mittels {F}initer {E}lemente
  und ${\mathcal {\char 72}}$-{M}atrizen}, Ph.D. thesis, Technische
  Universit{\"a}t M{\"u}nchen, 2006.

\bibitem[SS11]{SauterSchwab}
S.A. Sauter and Ch. Schwab, \emph{Boundary element methods}, Springer Series in
  Computational Mathematics, vol.~39, Springer-Verlag, Berlin, 2011.

\bibitem[Ste70]{stein70}
E.M. Stein, \emph{Singular integrals and differentiability properties of
  functions}, Princeton University Press, 1970.

\bibitem[Ste08]{Steinbach}
O.~Steinbach, \emph{Numerical approximation methods for elliptic boundary value
  problems}, Springer, New York, 2008.

\bibitem[SZ90]{ScottZhang}
L.~R. Scott and S.~Zhang, \emph{Finite element interpolation of nonsmooth
  functions satisfying boundary conditions}, Math. Comp. \textbf{54} (1990),
  no.~190, 483--493.

\bibitem[Tau03]{tausch03}
J.~Tausch, \emph{Sparse {BEM} for potential theory and {S}tokes flow using
  variable order wavelets}, Comput. Mech. \textbf{32} (2003), no.~4-6,
  312--318.

\bibitem[TW03]{tausch-white03}
J.~Tausch and J.~White, \emph{Multiscale bases for the sparse representation of
  boundary integral operators on complex geometry}, SIAM J. Sci. Comput.
  \textbf{24} (2003), no.~5, 1610--1629.

\bibitem[Tyr00]{tyrtyshnikov00}
E.E. Tyrtyshnikov, \emph{Incomplete cross approximation in the mosaic-skeleton
  method}, Computing \textbf{64} (2000), no.~4, 367--380, International
  GAMM-Workshop on Multigrid Methods (Bonn, 1998).

\bibitem[vPSS97]{petersdorff-schwab-schneider97}
T.~von Petersdorff, Ch. Schwab, and R.~Schneider, \emph{Multiwavelets for
  second-kind integral equations}, SIAM J. Numer. Anal. \textbf{34} (1997),
  no.~6, 2212--2227.

\bibitem[Wah91]{wahlbin91}
L.~Wahlbin, \emph{Local behavior in finite element methods}, Handbook of
  numerical analysis. Volume II: Finite element methods (Part 1) (P.G. Ciarlet
  and J.L. Lions, eds.), North Holland, 1991, pp.~353--522.

\end{thebibliography}
\bibliographystyle{amsalpha}

\end{document}